\documentclass[11pt]{amsart}
\usepackage{amsfonts,latexsym,amsthm,amssymb,amsmath,amscd,euscript,tikz, tikz-cd}
\usepackage[alphabetic, msc-links, bibtex-style, nobysame]{amsrefs}
\usepackage{stackengine}
\usepackage{framed}
\usepackage{empheq}
\usepackage{xfrac}
\usepackage[makeroom]{cancel}
\usepackage{faktor}
\usepackage{braket}
\usepackage{pgf,tikz,pgfplots}
\usepackage{tikz-3dplot}
\usetikzlibrary{3d,calc}
\usepgfplotslibrary{groupplots}
\pgfplotsset{compat=1.18}
\usepgfplotslibrary{fillbetween} 
\usepackage{mathrsfs}
\usetikzlibrary{arrows}
\definecolor{wrwrwr}{rgb}{0.3803921568627451,0.3803921568627451,0.3803921568627451}
\definecolor{rvwvcq}{rgb}{0.08235294117647059,0.396078431372549,0.7529411764705882}
\definecolor{mblue}{rgb}{0.2, 0.3, 0.8}
\definecolor{morange}{rgb}{1, 0.5, 0}
\definecolor{mgreen}{rgb}{0.1, 0.4, 0.2}
\definecolor{mred}{rgb}{0.5, 0, 0}
\definecolor{ForestGreen}{RGB}{34,139,34}
\usepackage{float}

\numberwithin{equation}{section}

\usepackage{lmodern}
\usepackage{supertabular}
\usepackage{amssymb}
\usepackage{enumerate}
\usepackage{stmaryrd}
\usepackage{bbm}
\usepackage{mathtools}
\usepackage{setspace}
\usepackage{tikz,tikz-cd}
\usetikzlibrary{matrix,calc,positioning,arrows,decorations.pathreplacing,patterns,knots}

\newcommand{\la}{\langle}
\newcommand{\rg}{\rangle}

\newtheorem{theorem}{{Theorem}}[section]
\newtheorem*{theorem*}{Theorem}
\newtheorem{lemma}[theorem]{Lemma}
\newtheorem{conjecture}[theorem]{Conjecture}
\newtheorem{proposition}[theorem]{Proposition}

\newtheorem{corollary}[theorem]{Corollary}
\newtheorem*{corollary*}{Corollary}

\theoremstyle{definition}

\usepackage{lmodern,url,enumerate,mathtools}
\usepackage[hmargin = 1in,vmargin=1in]{geometry}
\usepackage{graphicx}
\usepackage{subcaption}

\usepackage{hyperref}
    \hypersetup{colorlinks=true,citecolor=ForestGreen,linkcolor = blue,urlcolor =black,linkbordercolor={1 0 0}}

\newcommand{\ve}{\varepsilon}

\newcommand{\cA}{\mathcal{A}}
\newcommand{\cC}{\mathcal{C}}

\newcommand{\cH}{\mathcal{H}}
\newcommand{\cJ}{\mathcal{J}}
\newcommand{\cL}{\mathcal{L}}

\newcommand{\cR}{\mathcal{R}}
\newcommand{\cS}{\mathcal{S}}\newcommand{\cT}{\mathcal{T}}

\newcommand{\bN}{\mathbb{N}}

\newcommand{\bR}{\mathbb{R}}
\newcommand{\bS}{\mathbb{S}}

\newcommand{\nc}{\newcommand}

\nc{\on}{\operatorname}
\nc{\p}{\partial}
\nc{\ol}{\overline}
\nc{\ul}{\underline}
\nc{\pa}{\partial}

\nc{\pb}{\partial_b}
\nc{\pc}{\partial_c}
\nc{\pd}{\partial_d}
\nc{\pe}{\partial_e}
\nc{\pf}{\partial_f}
\nc{\pg}{\partial_g}
\nc{\ph}{\partial_h}
\nc{\pari}{\partial_i}
\nc{\pj}{\partial_j}
\nc{\pk}{\partial_k}
\nc{\pl}{\partial_l}
\nc{\pell}{\partial_\ell}
\nc{\parm}{\partial_m}
\nc{\pn}{\partial_n}
\nc{\po}{\partial_o}
\nc{\pp}{\partial_p}
\nc{\pq}{\partial_q}
\nc{\pr}{\partial_r}
\nc{\ps}{\partial_s}
\nc{\pt}{\partial_t}
\nc{\pu}{\partial_u}
\nc{\pv}{\partial_v}
\nc{\pw}{\partial_w}
\nc{\px}{\partial_x}
\nc{\py}{\partial_y}
\nc{\pz}{\partial_z}

\usepackage{booktabs}

\numberwithin{equation}{section}
\makeatletter
\@addtoreset{equation}{section}
\makeatother

\allowdisplaybreaks

\usepackage{accents}
\renewcommand{\underbar}[1]{\underaccent{\bar}{#1}}

\title{Stability inequalities for one-phase cones}
\date{\today}

\author[Benjy Firester]{Benjy Firester}
\address{
Department of Mathematics, MIT \newline 
{\href{mailto:benjyfir@mit.edu}{benjyfir@mit.edu}}}
\author[Raphael Tsiamis]{Raphael Tsiamis}
\address{
Department of Mathematics, Columbia University \newline
{\href{mailto:r.tsiamis@columbia.edu}{r.tsiamis@columbia.edu}}}
\author[Yipeng Wang]{Yipeng Wang}
\address{
Department of Mathematics, Columbia University \newline
{\href{mailto:yw3631@columbia.edu}{yw3631@columbia.edu}}}

\begin{document}

\begin{abstract}
We obtain strict stability inequalities for homogeneous solutions of the one-phase Bernoulli problem.
We prove that in dimension $7$ and above, cohomogeneity one solutions with bi-orthogonal symmetry are strictly stable.
As a consequence, we obtain a bound on the first eigenvalue and the decay rates of Jacobi fields, with applications to the generic regularity of the one-phase problem.
\end{abstract}

\maketitle

\vspace*{-0.25in}


\section{Introduction}

We prove the strict stability of a class of homogeneous solutions to the one-phase free boundary problem.
The one-phase Bernoulli free boundary problem models steady configurations of an incompressible inviscid fluid with a free surface.
Such configurations arise as critical points of an energy functional consisting of two components: a kinetic term and a volume term encoding the effect of the exterior pressure, defined via the following model.
Given an open domain $B$ in Euclidean space $\bR^n$ and a non-negative function $u$ on $B$, the \textit{Alt-Caffarelli functional} is defined as  
\begin{equation}\label{eqn:alt-caffarelli}
    \cJ(u, B) := \int_B |\nabla u|^2 + \chi_{ \{ u > 0 \} } .
\end{equation}
A critical point of the Alt-Caffarelli functional solves the \textit{one-phase free boundary problem}
\begin{equation}\label{eqn:one-phase-problem}
\Delta u = 0 \quad \text{in } \; \{ u > 0 \}, \qquad |\nabla u | = 1 \quad \text{on } \; \partial \{ u > 0 \}.
\end{equation}
In addition to its description of fluids and its significance in partial differential equations, the one-phase problem also has many known connections to the theory of minimal surfaces~\cites{traizet, jerison-kamburov}.
The Alt-Caffarelli functional has therefore attracted great interest for many decades, with foundational work done by Caffarelli, Jerison, Kenig, and others~\cites{alt-caffarelli, caffarelli-jerison-kenig, caffarelli-friedman, caffarelli-salsa, weiss}.

A fundamental question regarding the Alt-Caffarelli functional is the regularity of minimizing and stable solutions of the problem~\eqref{eqn:one-phase-problem}.
Critical points of $\cJ$ satisfy a monotonicity formula, due to Weiss~\cite{weiss}, which is saturated by homogeneous solutions (cones).
Therefore, taking rescalings of a given stable (resp.~minimizing) solution at a potential singular point produces a stable (resp.~minimizing) cone; together with a dimension-reduction argument, this connects the regularity theory for stable critical points with the study of stable homogeneous solutions.

In this paper, we settle the question of stability for all cohomogeneity one solutions of the one-phase problem with \textit{bi-orthogonal symmetry}, i.e., under an orthogonal group action of $O(n-k) \times O(k)$ on $\bR^n$.
These are the one-phase analogues of Lawson cones from minimal surface theory~\cite{lawson}.
Analogous estimates for Lawson cones were obtained by De~Philippis-Maggi~\cite{sharp-stability-plateau}, with Z.~Liu addressing a number of exceptional cones~\cite{lawson-liu}.

\begin{theorem}\label{thm:stability-of-one-phase-cones}
    For every $n \geq 3$ and $1 \leq k \leq n-2$, there exists a unique homogeneous $O(n-k) \times O(k)$-invariant solution of the one-phase problem, given by
    \[
    U_{n,k}(x,y) := c_{n,k}\sqrt{|x|^2 + |y|^2} \cdot f_{n,k} \left( \frac{|y|}{\sqrt{|x|^2 + |y|^2}} \right), \qquad f_{n,k}(t) := {}_2 F_1 \left( \frac{n-1}{2}, - \frac{1}{2} ; \frac{k}{2} ; t^2 \right)
    \]
    with free boundary given by the cone $\Gamma_{n,k} = \{ (x,y) : |y| \leq t_{n,k} \sqrt{|x|^2+ |y|^2} \}$.
    For every $n \leq 6$ and $1 \leq k \leq n-2$, this solution is unstable; for every $n \geq 7$ and $1 \leq k \leq n-2$, it is strictly stable.
\end{theorem}

The proof of Theorem~\ref{thm:stability-of-one-phase-cones} involves a detailed analysis of the linearized operator for the one-phase solutions $U_{n,k}$.
In particular, this involves studying inequalities and other properties of the hypergeometric functions $f_{n,k}(t)$.
Previously, certain special cases of Theorem~\ref{thm:stability-of-one-phase-cones} were computed in low dimensions.
Caffarelli-Jerison-Kenig, following their proof that stable homogeneous one-phase solutions are flat in dimension $3$, formulated a criterion for the stability of axisymmetric cones.
Studying the examples $U_{n,1}$, they numerically verified their instability, for $n \leq 6$, and stability for $7 \leq n \leq 12$.
Later, Hong~\cite{hong-singular} considered homogeneous one-phase solutions with large symmetry groups via the Caffarelli-Jerison-Kenig (CJK) criterion and computed that all solutions $U_{n,k}$ with $n \leq 6$ are unstable.
Moreover, he verified the stability of the solution $U_{7,5}$, cf.~\cite{hong-singular}*{Example 9} (with $p = n-k-1$ in our notation).

Crucially, De~Silva-Jerison~\cite{desilva-jerison-cones} proved that the axisymmetric solution $U_{7,1}$ is minimizing for $\cJ$.
Combined with the work of Jerison-Savin~\cite{jerison-savin} on the flatness of stable cones in dimensions $n \leq 4$, this proves that the critical dimension where a singular minimizer of the Alt-Caffarelli functional can appear satisfies $5 \leq d^* \leq 7$.
In our companion paper~\cite{FTW-1}, we prove that the one-phase cones $U_{n,k}$ are (strictly) minimizing for $n \geq 7$, except possibly for twenty exceptional pairs $(n,k)$ with $7 \leq n \leq 19$ and $n-4 \leq k \leq n-2$.
In particular, $U_{7,2}$ exhibits a new example of a minimizing cone in dimension $7$.
The results of the present paper are crucial for that property: in the process of proving Theorem~\ref{thm:stability-of-one-phase-cones}, we obtain the following stronger conclusion.
\begin{proposition}\label{prop:strict-subsolution}
    For every $n \geq 7$ and $1 \leq k \leq n-2$, let $g_{n,k}(t) := {}_2 F_1 ( 1 , \frac{n-4}{2} ; \frac{k}{2} ; t^2)$.
    We define the function
    \[
    V_{n,k}(x,y) := ( |x|^2 + |y|^2)^{\frac{4-n}{2}} g_{n,k} \left( \frac{|y|}{\sqrt{|x|^2 + |y|^2}} \right).
    \]
    Then, $V_{n,k}$ is a subsolution of the linearized equation at the cone $U_{n,k}$.
\end{proposition}
To prove that the solution $U_{n,k}$ is minimizing, one seeks to construct a global subsolution of the problem~\eqref{eqn:one-phase-problem} lying below the cone $U_{n,k}$ (resp.~a supersolution lying above $U_{n,k}$).
It is natural to start this construction ``near infinity'' (i.e., when $\rho = \sqrt{|x|^2 + |y|^2}$ is sufficiently large), where the global sub- and supersolution conditions are well-approximated by the linearized problem.
We therefore define a subsolution profile $W^-$ (resp.~a supersolution profile $W^+$) near infinity by
\[
W_- := U_{n,k} - \rho^{\alpha_-} g_-(t), \qquad W_+ := U_{n,k} + \rho^{\alpha_+} g_+(t) ,
\]
with $\alpha_{\pm}$ chosen so that $\rho^{\alpha_{\pm}} g_{\alpha_{\pm}}(t)$ is a homogeneous harmonic function.
Then, $W_{\pm}$ will have the required barrier property near infinity if and only if $\rho^{\alpha_{\pm}} g_{\alpha_{\pm}}(t)$ is a subsolution of the linearized problem at $U_{n,k}$.
Next, the existence of global barrier functions reduces to the question of appropriately gluing a compact profile to the data given by $W_{\pm}$.
Certain exponents are more amenable than others to such a gluing; notably, Proposition~\ref{prop:strict-subsolution} shows that $\alpha = 4-n$ is a valid exponent near infinity, and with additional work one can promote $W_-$ to a global subsolution.
We refer the reader to the discussion of~\cite{FTW-1} for the details of this argument; see also~\cite{desilva-jerison-cones}.

\subsection{Generic regularity}\label{section:generic-regularity}

Proposition~\ref{prop:strict-subsolution} has additional significance because it produces a bound on the first eigenvalue and the decay rate of Jacobi fields on the cone $U_{n,k}$.
In~\cite{generic-regularity-II}, a key ingredient to the generic regularity of minimizers for the Alt-Caffarelli energy was a strictly negative dimensional upper bound for $\Lambda_n := \sup_{ U \in \cS \cC(\bR^n) } \lambda_1$, the infimum of the absolute value of the first stability eigenvalue as $U$ ranges over all singular minimizing cones in $\bR^n$, denoted $\cS \cC(\bR^n)$.
The quantities $\Lambda_n$ and 
\begin{equation}\label{eqn:gamma-n-definition}
\gamma_n := -  \frac{n-2}{2} +\sqrt{ \Bigl( \frac{n-2}{2} \Bigr)^2 + \Lambda_n } \in \Bigl[ - \frac{n-2}{2}  , 0 \Bigr)
\end{equation}
figure crucially in this theory.
~\cite{generic-regularity-II}*{Theorem 1.3} shows that any minimizing solution $U$ of~\eqref{eqn:alt-caffarelli} in dimension $d \geq d^* + 3$ (for $d^*$ the critical dimension admitting a singular minimizing cone) can be generically perturbed, in an appropriate sense, to nearby solutions $U_t$ which satisfy
\[
\dim_{\cH} \on{sing} (U_t) \leq d - d^* - 2 + \gamma_d .
\]
Therefore, it is important to understand the sharp values of $(\gamma_n, \Lambda_n)$ as well as the cones on which these are attained, cf.~\cite{generic-regularity-II}*{Open Questions 1 and 2}.
The minimal surface counterpart of this question was proven by Zhu~\cite{zhu-shrinkers}, who used techniques from mean curvature flow to prove the corresponding quantity is maximized by Lawson's cones $C(\bS^{n-k-1} \times \bS^k)$ independently of $k$.
Working in the class of $O(n-k) \times O(k)$-invariant one-phase cones $U_{n,k}$, we observe a different phenomenon: for fixed $n$, the quantity $\lambda_1(n,k)$ is strictly increasing in $k$, attaining its maximum at $k=n-2$.
It is therefore natural to formulate the conjecture
\begin{conjecture}\label{conj:eigenvalue}
The sharp maximum $\Lambda_n$ is attained by the cone $U_{n,n-2}$.
\end{conjecture}
Of course, the supremum $\Lambda_n$ is taken over all \textit{minimizing} cones in $\bR^n$, and $U_{n,1}$ was previously the only known example.
Our results in~\cite{FTW-1}*{Theorem 1.4} extend this minimality to all $U_{n,k}$ for $n \geq 20$ (and all but twenty pairs $(n,k)$ for $7 \leq n \leq 19$), therefore making Conjecture~\ref{conj:eigenvalue} well-posed.
Let us denote 
\[
\bar{\Lambda}_n := \lambda_1(n,n-2), \qquad \bar{\gamma}_n = - \tfrac{n-2}{2} + \sqrt{ \bigl( \tfrac{n-2}{2} \bigr)^2 + \bar{\Lambda}_n} \, , 
\]
meaning that we expect $\Lambda_n = \bar{\Lambda}_n$ and $\gamma_n = \bar{\gamma}_n$.
Computing the first few values of $\bar{\gamma}_n$ and $\bar{\Lambda}_n$ for $n \geq 7$ (cf.~Table~\ref{tab:eigenvalues}) indicates that the quantity $\bar{\Lambda}_n$ is strictly decreasing in $n$, the quantity $\bar{\gamma}_n$ is strictly increasing in $n$, and $\bar{\gamma}_n \in (-2,-1)$ for all $n$.
That is, the constants $\bar{\gamma}_n$ (conjecturally equal to $\gamma_n$) are expected to play the same role as the spectral quantity $\alpha_n$ used in establishing the generic regularity of area-minimizing hypersurfaces, cf.~\cite{cmsw-11}*{Lemma 5.3}.

Providing evidence towards these properties, we obtain the following bounds involving the indicial roots $\gamma_{\pm}(n,k)$ of the cones $U_{n,k}$, defined in~\eqref{eqn:fast-slow-decay}, which correspond to the ``rapid'' and ``slow'' decay for one-phase solutions lying on one side of the cone and asymptotic to $U_{n,k}$ at infinity.
\begin{proposition}\label{prop:first-eigenvalue-bound}
    The first eigenvalue $\lambda_1(n,k)$ and decay rates $\gamma_{\pm}(n,k)$ on the cone $U_{n,k}$ satisfy
    \begin{equation}\label{eqn:lambda-1-n-k-bound}
        \lambda_1(n,k) > 8 - 2n \qquad \text{and} \qquad \gamma_-(n,k) \in (2-n,4-n), \quad \gamma_+(n,k) \in (-2,0).
    \end{equation}
    Notably, we have $(4-n,-2) \subset (\gamma_- , \gamma_+)$, meaning that there do not exist Jacobi fields (solutions of the linearized problem~\eqref{eqn:linearized-problem}) on $U_{n,k}$ with decay $\sim \rho^{\sigma}$, for any $\sigma \in (4-n,-2)$.
\end{proposition}
Let us mention another important application of our results.
Recently, Engelstein-Restrepo-Zhao~\cite{one-phase-simon-solomon} obtained infinitesimal and asymptotic rigidity results for minimizing one-phase cones that are strictly stable and \textit{strongly integrable through rotations} (see~\cite{one-phase-simon-solomon}*{Definition~2.5}) by analogy with the results of Allard-Almgren and Simon-Solomon for minimal cones~\cite{SimonSolomon}.
For the De~Silva-Jerison cones that are known to be minimizing and strictly stable for $n \geq 7$, due to~\cites{caffarelli-jerison-kenig, desilva-jerison-cones}, they obtain the strong integrability condition required to apply their theory, cf.~\cite{one-phase-simon-solomon}*{\S~2.2}.
Using the strict minimality of the cones $U_{n,k}$ that we obtain in~\cite{FTW-1}*{Theorem 1.4}, an adaptation of the linear spectrum analysis of~\cite{one-phase-simon-solomon}*{\S~2.2} following our Proposition~\ref{prop:spectral-analysis} would then imply the rigidity result~\cite{one-phase-simon-solomon}*{Theorem 1.3} for all such cones.

\subsection{Applications to other problems}

The results of Theorem~\ref{thm:stability-of-one-phase-cones} and Proposition~\ref{prop:strict-subsolution} have important applications to a range of related free boundary problems in geometry and PDE.

\smallskip \noindent \textbf{The capillary problem.}
Given a function $u : \bR^n \to \bR$, the $\theta$-\textit{capillary energy} of the epigraph $E_+(u) = \{ (x', x_{n+1}) : x_{n+1} > u(x') \} \subset \bR^{n+1}_+$ is given by
\[
\cA^{\theta} (u) = \int_{ \{ u > 0 \} \subset \bR^n } \Bigl( \sqrt{1 + |\nabla u|^2} - \cos \theta \, \chi_{ \{ u > 0 \} } \Bigr) \, dx'.
\]
The functional $\cA^{\theta}$ is a modification of the usual area functional, defined over the positive phase of $u$, with a penalty term determined by the angle $\theta$.
The critical points of $\cA^{\theta}$ produce \textit{capillary minimal graphs}.
A standard computation (see, for example,~\cite{improved-regularity}*{(1.2) and Lemma~4.7}) produces
    \begin{equation}\label{eqn:capillary-one-phase}
        \cA^{\theta}(u) = \frac{1}{2} \tan^2 \theta \int_{ \{ u_i >0 \} } \bigl( |\nabla \tilde{u}|^2 + \chi_{ \{ \tilde{u} > 0 \} } + O(\theta) \bigr) \, dx', \qquad \tilde{u} := \frac{1}{\tan \theta} u.
    \end{equation}
Consequently, the functionals $\cA^{\theta}(\cdot \frac{1}{\tan \theta})$ $\Gamma$-converge to the Alt-Caffarelli functional $\cJ$ as $\theta \downarrow 0$, allowing us to extract a minimizing solution of the one-phase problem as a subsequential limit of $\cA^{\theta_i}$-minimizers for $\theta_i \downarrow 0$.
Conversely, a sequence of capillary minimal graphs sufficiently $C^1$-close to a strictly stable solution will also be strictly stable, for $\theta$ sufficiently small.
This approach is taken in~\cite{FTW-1}, where $O(n-k) \times O(k)$-invariant capillary minimal cones are constructed for any angle $\theta$, converging (after rescaling) to the solution $U_{n,k}$ of Theorem~\ref{thm:stability-of-one-phase-cones} as $\theta \downarrow 0$.
Our results then imply that for small $\theta$, these solutions are also unstable when $n \leq 6$ and stable when $n \geq 7$.

\smallskip \noindent \textbf{The fractional one-phase problem.}
The thin one-phase problem is a generalization of the one-phase problem to fractional exponents $s \in (0,1)$.
The associated functional is formed by replacing the Dirichlet energy in~\eqref{eqn:alt-caffarelli} by the $H^s$-fractional seminorm of order $s \in (0,1)$, forming
\begin{equation}\label{eqn:thin-one-phase-functional}
    \cT_s(u, B) := [u]^2_{H^s(B)} + | \{ u > 0 \} \cap B |.
\end{equation}
The fractional seminorm $[v]^2_{H^s(\bR^n)}$ accounts for long-range interactions and is given by
\[
    [v]^2_{H^s(\bR^n)} := \frac{c_{n,s}}{2} \iint_{\bR^n \times \bR^n} \frac{(v(x) - v(y))^2}{|x-y|^{n+2s}} \, dx \, dy, \qquad \text{where } \; c_{n,s} = \frac{s 2^{2s} \Gamma ( \frac{n+2s}{2})}{\pi^{\frac{n}{2}} \Gamma(1-s)},
\]
where the constant $c_{n,s}$ comes from the definition of the $s$-fractional Laplacian
\[
(-\Delta)^s u(x) = c_{n,s} \, \text{P.V.} \int_{\bR^n} \frac{u(x) - u(y)}{|x-y|^{n+2s}} \, dy \, .
\]
The classical thin one-phase free boundary problem is obtained by considering $s = \frac{1}{2}$.
A variational argument (see, for example,~\cites{fractional-laplacian-problems, stable-thin-one-phase}) shows that a local minimizer (critical point) $u$ of $\cT_s$ solves the following problem on the domain $\Omega = \{ u > 0 \}$ (assumed to be sufficiently smooth), involving a condition on the fractional derivative of $u$ and $d(x) := \text{dist}(x,\partial \Omega)$:
\begin{equation}\label{eqn:thin-problem-conditions}
\left\{
\begin{aligned}
(-\Delta)^s u &= 0 && \text{in } \Omega \cap B, \\
u &= 0 && \text{in } \Omega^c \cap B, \\
\Gamma(1+s)\,\frac{u}{d^s} &= 1 && \text{on } \partial\Omega \cap B.
\end{aligned}
\right.
\end{equation}
As the fractional parameter $s$ increases to $1$, the thin one-phase problem $\Gamma$-converges to the standard Bernoulli problem.
Fern\'andez-Real and Ros-Oton~\cite{stable-thin-one-phase} formulated a stability criterion for the thin one-phase problem which, in the case of solutions with cohomogeneity one (in particular, with bi-orthogonal symmetry $O(n-k) \times O(k)$ as the ones examined here) reduces to a one-dimensional inequality.
The resulting condition is related to the Mellin transform operator associated with the Dirichlet-to-Neumann map of an appropriate radial function, generalizing the approach of Caffarelli-Jerison-Kenig~\cites{cjk-2}.
Moreover, they proved the instability of axisymmetric free boundaries in dimension $d \leq 5$ for all $s \in (0,1)$.
As $s \uparrow 1$, the stability inequality of~\cite{stable-thin-one-phase} recovers the criterion of Caffarelli-Jerison-Kenig.

\smallskip \noindent \textbf{The Alt-Phillips problem.}
Generalizing the Alt-Caffarelli functional, the \textit{one-phase Alt-Phillips} free boundary problem concerns the study of non-negative minimizers of the functional
\[
\cJ_{\gamma}(u) := \int_{B_1} |\nabla u|^2 + u^{\gamma} \chi_{ \{ u > 0 \} }, \qquad \text{where } \; \gamma \in (-2,2)
\]
among functions with non-negative boundary datum on $\partial B_1$, cf.~\cite{alt-phillips}.
For $\gamma = 0$, one recovers the standard one-phase problem, whereas $\gamma = 1$ corresponds to the \textit{obstacle problem}.
Recently, De~Silva-Savin studied the Alt-Phillips problem for negative exponents $\gamma$ and proved that the functional $\cJ_{\gamma}(u)$ $\Gamma$-converges to the Alt-Caffarelli energy, for $\gamma \uparrow 0$, and to the perimeter functional for $\gamma \downarrow -2$.
For this problem, the stability condition was obtained by Karakhanyan-Sanz-Perela, in the case of positive exponent $\gamma>0$, and by Carducci-Tortone, for negative exponents~\cites{stable-cones-alt-phillips, carducci-tortone}.
For this problem, the fractional exponent introduces a bulk term in the stability inequality, making an argument in the style of Jerison-Savin~\cite{jerison-savin} elusive.
However, sending $\gamma \downarrow 0$ (for positive $\gamma$) again reduces stability to the Caffarelli-Jerison-Kenig condition.

The problems discussed above depend continuously on a parameter and approximate the one-phase problem.
The operators corresponding to the capillary problem with respect to the angle $\theta$, the fractional one-phase problem with exponent $s$, and the Alt-Phillips problem with parameter $\gamma$ all have invertible linearizations at the value $\theta = 0$ (resp.~$s=1$ or $\gamma=0$) corresponding to the one-phase solution $U_{n,k}$.
Consequently, for some $\ve_n >0$ one may produce $O(n-k) \times O(k)$-invariant solutions $U_{n,k,\theta}$ of the capillary problem for angle $\theta \in (0,\ve_n)$ (resp.~$U_{n,k,s}$ for fractional one-phase exponents $s \in (1-\ve_n,1)$, or $U_{n,k,\gamma}$ for Alt-Phillips exponents $\gamma \in (-\ve_n, \ve_n)$).
We refer the reader to~\cite{FTW-1}*{Theorem 1.2}, where the existence and uniqueness of such solutions is obtained in greater generality.
Using Theorem~\ref{thm:stability-of-one-phase-cones} and the above discussion, one could then deduce that the solutions $U_{n,k,\theta}$ (resp.~$U_{n,k,s}$ and $U_{n,k,\gamma}$) are unstable for $n \leq 6$ and strictly stable for $n \geq 7$.
These connections will be addressed in greater detail in upcoming work on the capillary problem and related questions, see~\cite{FTW_Stability}.

\section{Preliminaries}
We collect here some important properties of stable critical points of the functional $\cJ$, which we will use to prove Theorem~\ref{thm:stability-of-one-phase-cones}.
We will also introduce some properties of hypergeometric functions that are important for studying the solutions $U_{n,k}$.

\subsection{Stability of solutions}\label{subsection:stability-of-solutions}

A critical point $U$ of $\cJ$ is called \textit{stable}, if the second variation of the functional $\cJ$ at $U$ is non-negative on every annulus $B^n_{r_2} \setminus \bar{B}^n_{r_1}$ for $0 < r_1 < r_2$.
The positivity of the second variation with respect to inner perturbations in an annulus $B^n_{r_2} \setminus \bar{B}^n_{r_1}$ was computed in~\cite{caffarelli-jerison-kenig}, and amounts to
\[
\int_{B^n_{r_2} \setminus \bar{B}^n_{r_1}} |\nabla \phi|^2 \geq \int_{ \partial ( B^n_{r_2} \setminus \bar{B}^n_{r_1} )} H \phi^2, \qquad \text{for any } \; \phi \in C^2(B^n_{r_2} \setminus \bar{B}^n_{r_1}).
\]
where $H$ denotes the mean curvature of the free boundary $\partial \{ U > 0 \}$ with respect to the outward-pointing normal vector.
For the one-phase problem, it is well-known that $H>0$ when $U$ is not the half-space solution, meaning that the free boundary is strictly mean-convex.

The stability of a homogeneous solution $U$ can be recast as a variational problem on the spherical link of the positive phase, i.e., the spherical domain $\Omega_S := \{ U > 0 \} \cap \bS^{n-1}$.
The first eigenvalue is computed via the minimization
\begin{equation}\label{eqn:Lambda-variational}
\lambda_1 := \inf_v \frac{\int_{\Omega_S} |\nabla v|^2 - \int_{\partial \Omega_S} H v^2}{\int_{\Omega_S} v^2}
\end{equation}
and the solution $U$ is stable (in any bounded domain that avoids the origin) if and only if $\lambda_1 \geq - ( \frac{n-2}{2})^2$.
We say that $U$ is \textit{strictly stable} if this inequality is strict, meaning that $\lambda_1 > - ( \frac{n-2}{2})^2$.
Since $H>0$, testing~\eqref{eqn:Lambda-variational} with $v=1$ shows that $\lambda_1 < 0$.

The stability of the solution $U$ is equivalent to the existence of a subsolution for the linearization of the one-phase problem at $U$.
The linearized problem at the solution $U$ assumes the form
\begin{equation}\label{eqn:linearized-problem}
\begin{cases}
    \Delta \varphi = 0 & \text{in } \; \{ U > 0 \}, \\
    \partial_{\nu} \varphi + H \varphi = 0 & \text{on } \; \partial \{ U > 0 \}.
\end{cases}
\end{equation}
where $\nu$ denotes the normal vector to $\partial \{ U > 0 \}$ pointing into $\{ U > 0 \}$.
Accordingly, the cone $U$ is \textit{strictly stable} if there exists a strict subsolution to the above problem.
We refer the reader to~\cites{jerison-savin, hardt-simon-DSJ, one-phase-simon-solomon} for a detailed discussion of (strict) stability properties and their implications.

Since $\Omega = \{ U > 0 \}$ is a cone, applying a separation of variables allows us to study solutions of~\eqref{eqn:linearized-problem} having the form $\varphi = \rho^{\alpha} g(\omega)$, for a function $g$ defines on the spherical link $\Omega_S$.
We decompose the Euclidean gradient in polar coordinates as $\nabla = \omega \partial_{\rho} + \rho^{-1} \nabla^{\bS^{n-1}}$ and observe that
\[
\partial_{\rho} \varphi = \alpha \rho^{\alpha - 1} g(\omega), \qquad \nabla^{\bS^{n-1}} \varphi = \rho^{\alpha} \nabla^{\bS^{n-1}} g(\omega).
\]
To compute the boundary term, note that $U$ is a homogeneous solution and $\{ U > 0 \}$ is a cone, so the free boundary $\partial \{ U > 0 \}$ contains the radial direction $\omega$ in its tangent space.
Therefore, the unit normal $\nu$ is orthogonal to $\omega$, meaning that $\la \omega, \nu \rg = 0$ on $\partial \{ U > 0 \} \setminus \{ 0 \}$.
Since $|\nabla U| = 1$ for the one-phase solution and $\nabla U$ is normal to the level set $\partial \{ U > 0 \}$, we have $\nabla U = \nu$ along the free boundary.
We therefore obtain
\begin{equation}\label{eqn:normal-derivative-angular}
    \partial_{\nu} \varphi = \la \nabla \varphi , \nu \rg = \partial_{\rho} \varphi \la \omega, \nu \rg + \rho^{-1} \la \nabla^{\bS^{n-1}} \varphi , \nu \rg =   \rho^{\alpha - 1} \la \nabla^{\bS^{n-1}} g(\omega), \nabla U \rg .
\end{equation}
The solutions of~\eqref{eqn:linearized-problem} therefore correspond to eigenfunctions of the spectral problem
\begin{equation}\label{eqn:Robin-problem}
    \begin{cases}
        - \Delta_{\bS^{n-1}} g = \lambda g & \text{in } \; \Omega_S, \\
        \la \nabla^{\bS^{n-1}} g, \nu \rg + H g = 0 & \text{on } \; \partial \Omega_S,
    \end{cases}
\end{equation}
where the eigenvalue $\lambda$ is related to the homogeneity degree $\alpha$ by 
\begin{equation}\label{eqn:degree-of-homogeneity}
    \alpha = - \frac{n-2}{2} \pm \sqrt{ \Bigl( \frac{n-2}{2} \Bigr)^2 + \lambda} \, .
\end{equation}
The solutions of~\eqref{eqn:Robin-problem} are eigenfunctions of the spherical Laplacian on $\Omega_S$ with Robin boundary condition.
This operator has discrete spectrum, given by the eigenfunctions
\[
\lambda_1 < \lambda_2 \leq \cdots \to + \infty \, .
\]
We may therefore investigate the stability of a solution $U$ by taking $\varphi = \rho^{\alpha} g(\omega)$ to be a harmonic function with $\alpha = \frac{2-n}{2}$ and examining the resulting boundary term of~\eqref{eqn:degree-of-homogeneity}.
The homogeneous solution $U$ is unstable, (borderline) stable, or strictly stable if the resulting boundary term of~\eqref{eqn:degree-of-homogeneity} is strictly positive, zero, or strictly negative, respectively.
This reformulation of the stability condition was used in~\cite{jerison-savin}*{Proposition 2.2} and~\cite{caffarelli-jerison-kenig}*{\S~3}.

When the homogeneous solution $U$ is a strictly stable global minimizer, De Silva, Jerison, and Shahgholian \cite{hardt-simon-DSJ}*{Theorem 1} proved the existence of a unique pair $\bar{U}, \underbar{U}$ of global minimizers with analytic free boundaries at distance $1$ from the origin, such that $\underbar{U} \leq U \leq \bar{U}$ and the scalings of the two solutions foliate $\bR^{n+1}_+ = \bR^n \times [0,\infty)$.
As in~\eqref{eqn:degree-of-homogeneity}, we define
\begin{equation}\label{eqn:fast-slow-decay}
 \gamma_{\pm} :=  - \frac{n-2}{2} \pm \sqrt{ \Bigl( \frac{n-2}{2} \Bigr)^2 + \lambda_1}   
\end{equation}
for $\lambda_1$ the first eigenvalue of the cone $U$; then, for $p \in  \bR^n \setminus B_R(0)$ and $\rho := |p|$ we may write
\begin{equation}\label{eqn:underbar-overbar}
\underbar{U} = U - \underbar{a} \rho^{\gamma_{\pm}} g_1(\omega) + O ( \rho^{\gamma_{\pm} - \alpha'}), \qquad \bar{U} = U + \bar{a} \rho^{\gamma_{\pm}} g_1(\omega) + O ( \rho^{\gamma_{\pm} - \alpha'}) 
\end{equation}
for $g_1$ the first eigenfunction and universal constants $R,\underbar{a}, \bar{a}, \alpha'$.

In terms of the dimensional quantities $\gamma_n$ and $\Lambda_n$ introduced in Section~\ref{eqn:gamma-n-definition}, we see that the ``slow'' indicial root (homogeneity) $\gamma_+$ gives rise to $\gamma_n$ via
\begin{equation}\label{eqn:gamma-n-sup}
    \gamma_n = \sup_{U \in \cS \cC(\bR^n)} \gamma_+
\end{equation}
so $\gamma_n$ arises as the supremum of the ``slow'' decay rates over all non-flat minimizing cones in $\bR^n$.

\subsection{\texorpdfstring{$O(n-k)\times O(k)$}{O(n-k)xO(k)}-invariant one-phase cones}

We study solutions of the one-phase problem in $\bR^n$ having $O(n-k) \times O(k)$-symmetry.
Let us work in Euclidean space $\bR^n = \bR^{n-k} \times \bR^k$ with coordinates $(x,y) \in \bR^{n-k} \times \bR^k$.
We define
\[
\rho := \sqrt{ |x|^2 + |y|^2}, \qquad t = \frac{|y|}{\rho} \in [0,1].
\]
\begin{lemma}\label{lemma:alpha-homogeneous}
    Given a real number $\alpha$, any $O(n-k) \times O(k)$-invariant, $\alpha$-homogeneous harmonic function in a wedge of $\bR^n$ is given by
    \[
    U_{\alpha}(x,y) = c \, \rho^{\alpha} g_{n,k,\alpha}(t), \qquad \text{for some } \; c \in \bR,
    \]
    where we define 
    \begin{equation}\label{eqn:general-hypergeometric-expression}
g_{n,k,\alpha}(t) = {}_2 F_1 \left( \frac{n + \alpha-2}{2} ,  - \frac{\alpha}{2}; \frac{k}{2} ; t^2 \right).
\end{equation}
\end{lemma}
\begin{proof}
    Let us write $U_{\alpha} = u(\rho, t)$, where $u(\rho,t) = \rho^{\alpha} g(t)$ is an $\alpha$-homogeneous function.
    Note that $g(t)$ must be an even function of $t$ in order for $U(x,y)$ to be well-defined in $x,y$.
    A standard computation shows that $\la \nabla \rho, \nabla t \rg = 0$ with
    \[
    |\nabla \rho|^2 = 1, \qquad |\nabla t|^2 = \frac{1-t^2}{\rho^2}, \qquad
    \Delta \rho = \frac{n-1}{\rho}, \qquad \Delta t = \frac{(k-1) t^{-1} - (n-1) t}{\rho^2}
    \]
    We therefore obtain, for any function $U(x,y) = u(\rho, t)$,
    \begin{equation}\label{eqn:DeltaU-computation}
        \begin{split}
             \Delta U &= u_{\rho \rho} |\nabla \rho|^2 + 2 u_{\rho t} \la \nabla \rho, \nabla t \rg + u_{tt} |\nabla t|^2 + u_{\rho} \Delta {\rho} + u_t \Delta t \\
    &= u_{\rho \rho} + \frac{n-1}{\rho} u_{\rho} + \frac{1-t^2}{\rho^2} u_{tt} + \frac{(k-1) t^{-1} - (n-1) t}{\rho^2} u_t.
        \end{split}
    \end{equation}
    Taking $u(\rho,t) = \rho^{\alpha} g(t)$, where $g(0) = 1$ up to scaling, we compute that
    \begin{equation}\label{eqn:laplace-rho-a-g}
    \Delta (\rho^{\alpha} g) = \rho^{\alpha-2} \left( (1-t^2) g'' + \bigl( (k-1) t^{-1} - (n-1) t \bigr) g' + \alpha (\alpha + n-2) g \right).
\end{equation}
The unique even solution of the ODE $\rho^{2-\alpha} \Delta(\rho^{\alpha} g) = 0$
from \eqref{eqn:laplace-rho-a-g} with $g(0)=1$ is given by~\eqref{eqn:general-hypergeometric-expression}.
\end{proof}

\begin{lemma}\label{lemma:homogeneous-solution-fn,k}
Consider the hypergeometric function $f_{n,k}(t) := {}_2 F_1 ( \frac{n-1}{2} , - \frac{1}{2} ; \frac{k}{2} ; t^2)$, with a zero at $t_{n,k} > 0$.
The unique $O(n-k) \times O(k)$-invariant homogeneous solution of the one-phase problem in $\bR^n$ is given by
\begin{equation}\label{eqn:one-phase-solution}
    U_{n,k}(x,y) = c_{n,k} \, \rho \, f_{n,k}(t), \qquad \text{where } \; c_{n,k} := \frac{1}{\sqrt{1 - t_{n,k}^2} \, |f'_{n,k}(t_{n,k})|} \, .
\end{equation}
\end{lemma}
\begin{proof}
    The computation of Lemma~\ref{lemma:alpha-homogeneous} specialized to $\alpha=1$ shows that the unique $O(n-k) \times O(k)$-invariant homogeneous harmonic function in $\bR^n$ has the form $c \,\rho f_{n,k}(t)$, where
    \[
    f_{n,k} (t) = g_{n,k,1} (t) = {}_2 F_1 \Bigl( \frac{n-1}{2} , - \frac{1}{2} ; \frac{k}{2} ; t^2 \Bigr)
    \]
    by using~\eqref{eqn:general-hypergeometric-expression}.
    To determine the constant $c$, we observe that
    \[
    |\nabla U_{n,k}|^2 = (\partial_{\rho} U_{n,k})^2 + \frac{1-t^2}{\rho^2} ( \partial_t U_{n,k})^2 = c_{n,k}^2 \bigl[ f_{n,k}(t)^2 + (1-t^2) f'_{n,k}(t)^2 \bigr].
    \]
    Along the free boundary $\partial \{ U_{n,k}> 0 \} = \{ \pm t_{n,k} \}$, the condition $|\nabla u| = 1$ from~\eqref{eqn:one-phase-problem} forces
    \[
    c_{n,k}^2 (1- t_{n,k}^2) f'_{n,k}(t_{n,k})^2 = 1
    \]
    which proves the equality~\eqref{eqn:one-phase-solution}.
    
\end{proof}
The case $k=1$ corresponds to the axisymmetric cones studied by De~Silva and Jerison~\cite{desilva-jerison-cones}.
The hypergeometric function $f_{n,k}(t)$ is the solution of $\cL_{n,k}f = 0$ for the operator
\begin{equation}\label{eqn:legendre-form}
\begin{split}
\cL_{n,k} f &= (1-t^2) f'' + (n-1) (f - t f') + (k-1) t^{-1} f' \\
&= \frac{1-t^2}{p(t)} \left[ (p f')' + (n-1) \frac{p}{1-t^2} f \right], \qquad \text{where } \; p(t) := t^{k-1}(1-t^2)^{\frac{n-k}{2}} \, .
\end{split}
\end{equation}
Using standard properties of hypergeometric functions, we observe that the function $g_{n,k,\alpha}$ of~\eqref{eqn:general-hypergeometric-expression} has a zero $t_{n,k,\alpha} \in (0,1)$ if and only if $\alpha \not\in [2-n,0]$, otherwise remains strictly positive on $[-1,1]$.
In fact, $\alpha \in (1-n,1)$ makes $g_{n,k,\alpha}(t) > 0$ in the interval $[-t_{n,k}, t_{n,k}]$ where $f$ is positive.
In the limit, $|t_{\alpha}| \uparrow 1$ as $\alpha \to 0^{\pm}$ or $\alpha \to (2-n)^{\pm}$, and $|t_{\alpha}| \downarrow 0$ as $|\alpha| \to \infty$.
Observe that such a solution of the problem~\eqref{eqn:Robin-problem} exists if and only if it corresponds to the first eigenvalue and eigenfunction, by Courant's nodal domain theorem.

In our situation, the cone $U_{n,k}$ has cohomogeneity one, therefore the spectral problem~\eqref{eqn:Robin-problem} can be reformulated as a simple Sturm-Liouville problem.
This will allow us to test the stability of cones by a simple adaptation of the Caffarelli-Jerison-Kenig criterion from~\cite{cjk-2}*{(12), (12)'}.

\begin{proposition}\label{prop:spectral-analysis}
    Consider the cone $U_{n,k}$ of Theorem~\ref{thm:stability-of-one-phase-cones}.
    A number $\lambda \in \bR$ is an eigenvalue of the linearized problem~\eqref{eqn:Robin-problem} if and only if for some $(p,q) \in \bN_0 \times \bN_0$, the even solution of the ODE
    \begin{equation}\label{eqn:eigenvalue-condition}
        (1-t^2)\Phi'' + \left(\frac{k-1}{t} - (n-1)t\right)\Phi' + \left(\lambda -\frac{p(p + n-k-2)}{1-t^2} - \frac{q(q+k-2)}{t^2}\right)\Phi = 0
    \end{equation}
    normalized by $\Phi(0) = 1$ for $q = 0$ or $\lim_{t \to 0} t^{-q} \Phi(t) = 1$ when $q > 0$ and satisfies
    \begin{equation}\label{eqn:Phi-boundary-condition}
        \frac{\Phi'(t_{n,k})}{\Phi(t_{n,k})} = \frac{(n-2)t_{n,k} - (k-1)t_{n,k}^{-1}}{1-t^2_{n,k}}.
    \end{equation}
\end{proposition}
\begin{proof}
    This result specializes the computation of~\eqref{eqn:Robin-problem} to the solution $U_{n,k}$.
    For the boundary term, we compute $\nabla t = \rho^{-1}\nabla^{\bS^{n-1}} t$, meaning that $|\nabla^{\bS^{n-1}} t| = \sqrt{1-t^2}$; moreover, the inward-pointing unit normal $\nu$ to $\partial \{ U_{n,k} > 0 \} = \{ (\rho, t) : t = t_{n,k} \}$ is given by $\nu = -\frac{\nabla^{\bS^{n-1}}t}{|\nabla^{\bS^{n-1}}t|}$.
Consequently,
\[
\la \nabla^{\bS^{n-1}} g, \nu\rg = g' \la \nabla^{\bS^{n-1}}t,\nu\rg = - \sqrt{1 - t_{n,k}^2} \, g'(t_{n,k}) \quad \text{on } \; \partial \{ U_{n,k} > 0 \}.
\]
The mean curvature of the $O(n-k) \times O(k)$-invariant hypersurface $\partial \{ U_{n,k} > 0 \} \subset \bR^n$ has contributions from $(n-k-1)$ principal curvatures equal to $\frac{t_{n,k}}{\rho\sqrt{1 -t_{n,k}^2}}$ along the $O(n-k)$-invariant factor, $(k-1)$ principal curvatures equal to $\frac{\sqrt{1-t_{n,k}^2}}{\rho t_{n,k}}$ along the $O(k)$-invariant factor, and one principal curvature along the radial direction which is equal to zero, since the normal direction is invariant under radial rescaling.
We therefore find that
\begin{align*}
    \rho H &= (n-k-1)\frac{t_{n,k}}{\sqrt{1-t_{n,k}^2}}-(k-1)\frac{\sqrt{1-t_{n,k}^2}}{t_{n,k}} = \frac{(n-2) t_{n,k} - (k-1) t_{n,k}^{-1} }{\sqrt{1 - t_{n,k}^2} }.
\end{align*}
Consequently, the boundary term in~\eqref{eqn:Robin-problem} becomes
\begin{equation}\label{eqn:mean-curvature-computation}
\la \nabla^{\bS^{n-1}} g, \nu \rg + H g = - \sqrt{1 - t^2_{n,k}} \Bigl( g'(t_{n,k}) - \frac{(n-2) t_{n,k} - (k-1) t_{n,k}^{-1}}{1 - t_{n,k}^2} g(t_{n,k}) \Bigr).
\end{equation}
We apply the equivariant separation of variables $g(x,y) = \Phi(t) X_p ( \frac{x}{|x|}) Y_q ( \frac{y}{|y|})$ on the link, where $X_p$ and $Y_q$ are spherical harmonics satisfying
\begin{align*}
    -\Delta_{\bS^{n-k-1}}X_p = p(p +n-k-2)X_p, \quad \text{ and}  \quad-\Delta_{\bS^{k-1}}Y_q = q(q +k-2)Y_q,
\end{align*}
for $(p,q) \in \bN_0 \times \bN_0$.
In the notation of~\eqref{eqn:legendre-form}, the Laplace operator becomes self-adjoint with
\[
- \frac{1-t^2}{p(t)} \frac{d}{dt} \bigl( p(t) \Phi'(t) \bigr) + \left( \frac{p(p+n-k-2)}{1-t^2} + \frac{q(q+k-2)}{t^2} \right) \Phi(t) = \lambda \Phi(t),
\]
and expanding this expression produces~\eqref{eqn:eigenvalue-condition}.
To obtain the boundary condition~\eqref{eqn:Phi-boundary-condition}, observe that $\frac{g}{\Phi(t)}$ is a purely angular function and $\partial_{\nu}$ acts only on $t$, as discussed in~\eqref{eqn:normal-derivative-angular}.
Consequently, the computation $\la \nabla^{\bS^{n-1}}g , \nu \rg +Hg = 0$ as in~\eqref{eqn:mean-curvature-computation} produces the desired expression.
We note that $t \sim |y|$ as $t \downarrow 0$, so $g \sim \Phi(t)X_p Y_q$ behaves like $Q_q(r \omega)$, the homogenization of $Y_q$ which is a harmonic polynomial on $\bR^{k}$, so the condition on $\Phi(0)$ near 0 shows that this extends smoothly across $\{ t = 0\}$.
\end{proof}

\begin{corollary}\label{cor:stability-criterion}
    Let $f_{n,k}(t)$ be as in Theorem~\ref{thm:stability-of-one-phase-cones}, having a zero at $t_{n,k}>0$.
    Let $g_{n,k,\alpha}(t)$ denote the hypergeometric function~\eqref{eqn:general-hypergeometric-expression}.
    Then, the cone $U_{n,k}$ is strictly stable if and only if
    \begin{equation}\label{eqn:strict-stability-inequality}
        \frac{g'_{n,k,\alpha}(t_{n,k})}{g_{n,k,\alpha}(t_{n,k})} > \frac{(n-2) t_{n,k} - (k-1) t_{n,k}^{-1}}{1 - t_{n,k}^2}
    \end{equation}
    holds for $\alpha = \frac{2-n}{2}$.
\end{corollary}
\begin{proof}
As remarked above, the strict stability (resp.~borderline stability, instability) of the cone $U_{n,k}$ is equivalent to the boundary term of~\eqref{prop:spectral-analysis} being strictly negative (resp.~zero, positive) for the homogeneous harmonic function $\varphi = \rho^{- \frac{n-2}{2}} g(\omega)$.
Lemma~\ref{lemma:alpha-homogeneous} shows that the homogeneous harmonic function of degree $\alpha = \frac{2-n}{2}$ on $\{ U_{n,k}  >  0 \}$ is given by $\varphi = \rho^{\frac{2-n}{2}} g_{n,k,\alpha}(t)$, and since $g_{n,k,\alpha} > 0$, the computation~\eqref{eqn:mean-curvature-computation} implies
\begin{equation}\label{eqn:sign-boundary-term}
\text{sgn} \bigl( \la \nabla^{\bS^{n-1}} g, \nu \rg + H g \bigr) = - \on{sgn} \Bigl( \frac{g'_{n,k,\alpha}(t_{n,k})}{g_{n,k,\alpha}(t_{n,k})} - \frac{(n-2) t_{n,k} - (k-1) t_{n,k}^{-1}}{1 - t_{n,k}^2} \Bigr).
\end{equation}
We therefore obtained the claimed strict stability condition~\eqref{eqn:strict-stability-inequality}.
\end{proof}

\subsection{Hypergeometric functions}

For $O(n-k) \times O(k)$-invariant homogeneous solutions of the one-phase problem, Lemma~\ref{lemma:homogeneous-solution-fn,k} shows that the profile curve is described by a hypergeometric function.
For real numbers $a,b,c$ with $c>0$, the hypergeometric function ${}_2 F_1 ( a,b ; c ; s)$ is defined as the solution of Euler's differential equation
\[
s(1-s) F''(s) + \left[ c - (a+b+1) s \right] \, F'(s) - ab \, F(s) = 0
\]
with $F(0) = 1$ and $F'(0) = \frac{ab}{c}$.
This function has the convergent power series expansion
\[
{}_2F_1(a,b;c;s) = \sum_{m=0}^{\infty} \frac{(a)_m(b)_m}{(c)_m} \frac{s^m}{m!} = 1 + \frac{ab}{c} \frac{s}{1} + \frac{a(a+1)b(b+1)}{c(c+1)} \frac{s^2}{2!} + \cdots, \qquad |s|<1,
\]
where $(q)_m$ denotes the rising Pochhammer symbol $(q)_m := \frac{\Gamma(q+m)}{\Gamma(q)} = q(q+1) \cdots (q+m-1)$.
The function is symmetric in $a, b$; if $a$ or $b$ is a positive integer, then the series terminates and the function reduces to a polynomial.
Iterating the identity $(a)_{m+1} = a(a+1)_m$, one obtains
\begin{equation}\label{eqn:d-n-derivative}
    \frac{d^m}{ds^m}{}_2F_1 (a,b ; c ; s) = \frac{(a)_m (b)_m}{(c)_m} {}_2 F_1 (a+m,b+m;c+m;s).
\end{equation}
When $c>b>0$, the hypergeometric function admits the integral formula
\begin{equation}\label{eqn:hypergeometric-integral-formula}
    {}_2 F_1 (a,b; c; s) = \frac{\Gamma(c)}{\Gamma(b) \Gamma(c-b)} \int_0^1 \tau^{b-1} (1-\tau)^{c-b-1} (1 - \tau s)^{-a} \, d \tau.
\end{equation}
We refer the reader to~\cites{gauss} and~\cite{dtmf}*{Ch.~15} for some properties of these functions that we will use.
We remark that the odd solution of the operator $\cL_{n,k}f = 0$ from~\eqref{eqn:legendre-form} is given by
\begin{equation}\label{eqn:fn,k,odd}
    f_{n,k,\text{odd}}(t) = (\on{sgn} \, t)^{k-1} \cdot t^{2-k}{}_2 F_1 \left( \frac{n-k+1}{ 2}, \frac{1-k}{2} ; 2 - \frac{k}{2}; t^2 \right)
\end{equation}
where multiplying by $(\on{sgn} \, t)^{k-1}$ produces an odd extension across $t=0$ when $k$ is even.

In the case $k=1$, we have $\cL_{n,1}(t) = 0$, and the second solution is obtained by a reduction of order.
We therefore obtain
\begin{equation}\label{eqn:fn,1(t)}
    f_{n,1}(t) = {}_2 F_1 \left( \frac{n-1}{2} , - \frac{1}{2} ; \frac{1}{2} ; t^2 \right) = t \int^t s^{-2} (1-s^2)^{- \frac{n-1}{2}} \, ds
\end{equation}
with the integration constant chosen so that $f_{n,1}(0) = 1$.

For $k = 2m+1$ or $k=2m$, the resulting functions have the form
\begin{equation}\label{eqn:general0frm}
\begin{split}
    f_{n,2m+1}(t) &= \frac{P_{n,m}(t^2)}{t^{2m}} \arcsin t + \sqrt{1-t^2} \frac{Q_{n,m}(t^2)}{t^{2m}}, \\
    f_{n,2m}(t) &= \frac{\tilde{P}_{n,m}(t^2)}{t^{2m-1}} \on{atanh} t + \sqrt{1-t^2} \frac{\tilde{Q}_{n,m}(t^2)}{t^{2m-1}},
\end{split}
\end{equation}
where the $P_{n,m}, \tilde{P}_{n,m}, Q_{n,m}, \tilde{Q}_{n,m}$ are explicit polynomials of degree $\leq m$ with rational coefficients depending on $n$.
For instance, we compute
\begin{align*}
f_{7,1}(t) &= - \frac{15}{8} t \, \on{atanh} t + \frac{ \frac{15}{8} t^4 - \frac{25}{8} t^2 + 1}{(1-t^2)^2}, \\
f_{8,1} (t) &= (1-t^2)^{- \frac{5}{2}} \left( 1-6t^2 + 8 t^4 - \frac{16}{5} t^6 \right).
\end{align*}
We will use the fact that $b = - \tfrac{1}{2}$ for the hypergeometric functions in our situation, whereby $(b)_m = (- \frac{1}{2})_m < 0$ for all $m \geq 1$.
This implies that $\tfrac{d^m}{ds^m} {}_2 F_1(a,b;c;s) < 0$ for the derivatives of all orders when $s<1$, $f_{n,k}(t)$ can be bounded by the polynomial
\begin{equation}\label{eqn:Fk(s)-cubic-bound}
f_{n,k}(t) \leq 1 - \frac{n-1}{2k} t^2 - \frac{n^2-1}{8k(k+2)} t^4 - \frac{(n^2-1)(n+3)}{16k(k+2)(k+4)} t^6
\end{equation}
which we will use to bound the roots $t_{n,k}$.

\section{Analysis of the linearized problem}

Using Corollary~\ref{cor:stability-criterion}, we reduce the stability of the one-phase cones $U_{n,k}$ to the validity of the inequality~\eqref{eqn:strict-stability-inequality}.
We will prove this inequality for all $n \geq 7$ and $1 \leq k \leq n-2$ by analyzing the hypergeometric functions $f_{n,k}(t)$.
We use the notation for $f_{n,k}$ and $g_{n,k,\alpha}$ introduced in~\eqref{eqn:general-hypergeometric-expression} and~\eqref{eqn:one-phase-solution}.
Note that $g_{n,k,\alpha}$ coincides with $f_{n,k}$ for $\alpha = 1$; in our setting for $g$, we focus on $\alpha \in (2-n,0)$.
The hypergeometric function $g_{n,k,\alpha}(t)$ is defined as the even solution of the ODE
\[
(1-t^2) g'' + \left( (k-1) t^{-1} - (n-1) t \right) g' + \alpha (\alpha + n-2) g = 0
\]
with initial conditions $g(0) = 1$ and $g'(0) = 0$.
We will suppress the dependence of $f,g$ on the parameters $n,k,\alpha$ when those are well-understood.

The main result of this section is that $\alpha = 4-n$ satisfies the inequality~\eqref{eqn:strict-stability-inequality} for all $n \geq 7$; this will prove Proposition~\ref{prop:strict-subsolution}.
For brevity, we will let ${\cA}_{n,k} \subset (2-n,0)$ denote the set of $\alpha \in (2-n,0)$ for which the condition \eqref{eqn:g-n,k-inequality} holds, so our claim is that $4 - n \in \cA_{n,k}$.
The argument treats a number of different cases, depending on the values of $\frac{k}{n} \in (0,1)$.

\subsection{Propagation of positivity}

First, we establish some results that allow us to propagate positivity backwards.
This will enable us to prove the inequality~\eqref{eqn:strict-stability-inequality} once certain estimates are in place.
In what follows, we suppress the dependence on $n,k$ and write $g_{\alpha} := g_{n,k,\alpha}$.

\begin{lemma}\label{lemma:riccati-g-is-positive}
    For any $\alpha \in (2-n,0)$, the function
    \[
    \frac{g'_{\alpha}(t)}{g_{\alpha}(t)} - \frac{(n-2) t - (k-1) t^{-1}}{1-t^2}
    \]
    has at most one zero in the interval $(0,1)$.
    If it is positive at some $t_* > 0$, then it is everywhere positive on $(0,t_*)$.
\end{lemma}
\begin{proof}
    We will work with the function
    \[
    \tilde{L}_g(t) := t (1-t^2) \frac{g'_{\alpha}(t)}{g_{\alpha}(t)} - (n-2) t^2 + (k-1)
    \]
    where $g_{\alpha}(t) := {}_2 F_1 \left( \frac{n+\alpha-2}{2}, - \frac{\alpha}{2} ; \frac{k}{2} ; t^2 \right)$.
    We first observe that $\tilde{L}_g(1) =-1$.
    Indeed, for $k < n-2$ we have $g_{\alpha}(t) \sim C (1-t^2)^{- \frac{n-k-2}{2}}$, so
    \[
    (1-t^2) \frac{g'_{\alpha}(t)}{g_{\alpha}(t)} = (1-t^2) \frac{d}{dt} \log g_{\alpha} \xrightarrow{t \to1} (1-t^2) \frac{n-k-2}{1-t^2} = n-k-2
    \]
    as $t \uparrow 1$.
    If $k=n-2$, then $g_{\alpha}(t) \sim C_0 + C_1 \log (1-t^2)$, so
    \[
    \frac{g'_{\alpha}(t)}{g_{\alpha}(t)} \sim -\frac{2t}{(1-t^2) \log (1-t^2)} \implies (1-t^2) \frac{g'_{\alpha}(t)}{g_{\alpha}(t)} \xrightarrow{t \to 1} 0
    \]
    as $t \uparrow 1$.
    In either case, we evaluate
    \[
    \lim_{t \uparrow 1} (1-t^2) \frac{g'_{\alpha}(t)}{g_{\alpha}(t)} = n-k-2
    \]
    as well as
    \[
    \left( (n-2)t- (k-1) t^{-1} \right) \Big\rvert_{t=1} = n-k-1.
    \]
    Combining the two computations, we obtain $\tilde{L}_g(1) = -1$.
    Using the differential equation satisfied by $g_{\alpha}$, as well as its associated Riccati form, we obtain by direct computation that
    \[
    t (1-t^2) \tilde{L}'_g(t) + \tilde{L}_g(t)^2 + (nt^2 - k) \tilde{L}_g(t) + P(t^2) = 0 ,
    \]
     where we define $P(s) := (n-2k) s + \alpha (\alpha + n-2) s (1-s) + (k-1)$.
     We introduce the function $L(s) := \tilde{L}_g(t)$, for $s = t^2$, which therefore satisfies the Riccati equation
     \begin{equation}\label{eqn:L-equation-this-one-promise}
         2 s (1-s) L'(s) + L(s)^2 + (ns-k) L(s) + P(s) = 0.
     \end{equation}
     In what follows, let us denote $\hat{\alpha} := \alpha ( \alpha+ n-2)$, for brevity, so $\hat{\alpha} \in ( - ( \frac{n-2}{2})^2, 0)$ for $\alpha \in (2-n,0)$ and we may write $P(s) = (n-2k) s + \hat{\alpha} s(1-s) + (k-1)$.
     The initial conditions for $g$ imply 
     \[
     L(0) = k-1, \qquad L'(0) = - (n-2) - \tfrac{\hat{\alpha}}{k}.
     \]
     Using the equation~\eqref{eqn:L-equation-this-one-promise}, we see that $\on{sgn} L'(s_*) = - \on{sgn} P(s_*)$ holds at any point where $L(s_*) = 0$.
     Since $\hat{\alpha}<0$, we have that $P(s)$ is a strictly convex quadratic function of $s$, with
     \[
     P(0) = k-1 \geq 0, \qquad P(1) = n-k-1 > 0.
     \]
     In particular, $P(s)$ has at most one sign change on $(0,1)$ if $k=1$, and at most two if $k>1$.

    We first consider the case $k=1$, so $L(0)=0$ and $L'(0) = - (n-2) - \hat{\alpha}$.
    If $L'(0) = 0$, then using~\eqref{eqn:L-equation-this-one-promise} gives $L''(0)<0$; therefore, $L'(0) \leq 0$ implies that $L<0$ for small $s>0$.
    Moreover,
    \[
    P(s) = (n-2) s + \hat{\alpha} s(1-s) = s \left( (n-2) + \hat{\alpha} - \hat{\alpha} s \right) = - s ( L'(0) + \hat{\alpha} s) > 0 \quad \text{for } \; s>0
    \]
    in this case, due to $L'(0) \leq 0$ and $\hat{\alpha}<0$.
    If there was a first instance $s_* \in (0,1)$ where $L(s_*) = 0$, we would have $L'(s_*) > 0$; however, this would force
    \[
    0 < \on{sgn} P(s_*)  = - \on{sgn} L'(s_*) < 0,
    \]
    leading to a contradiction.
    Consequently, $(n-2) + \hat{\alpha} \geq 0$ implies $L<0$ on $t \in (0,1)$.
    If $L'(0) = - ( n-2 + \hat{\alpha}) >0$, then the above computation shows that $L>0$  for small $s>0$, while $P<0$ for small $s>0$.
    At a first zero $s_* \in (0,1)$ of $L$, we must have $\on{sgn} L'(s_*) = - \on{sgn} P(s_*) <0$, meaning that the sign change of $L$ must occur after the sign change of $P$.
    If $L$ had another zero $\tilde{s}_* \in (s_*, 1)$, we would need $\on{sgn} L'(\tilde{s}_*) = - \on{sgn} P(\tilde{s}_*) \geq 0$, which is impossible since $P>0$ for $s>s_*$.
    Since $L(1) = -1$, we conclude that $L(s)$ has exactly one sign change on $(0,1]$ in this case.

    Combining the two cases, we conclude that $L(s)$ has at most one sign change on $(0,1]$ for $k=1$, and $L(1) = - 1$.
    Returning to $\tilde{L}_g(t) = L(t^2)$, we conclude that $\tilde{L}_g(t_*) > 0$ implies $\tilde{L}_g > 0$ on $(0, t_*]$, proving the claim in this case.
    As a byproduct of our argument, we obtain that $\tilde{L}_g(t_*) > 0$ is only possible for $L'(0) \geq 0$, which implies
    \[
    - L'(0) = (n-2) + \hat{\alpha} < 0 \iff (n-2) + \alpha (\alpha + n-2) < 0.
    \]
    Rearranging the inequality, we obtain that for any $t_* > 0$,
     \begin{equation}\label{eqn:a-subset-of}
        \left\{ \alpha : \frac{g'_{\alpha}(t_*)}{g_{\alpha}(t_*)} - \frac{(n-2) t_*}{1-t_*^2} \geq 0 \right\} \subset \left( \frac{2 - n - \sqrt{n^2 - 8n + 12}}{2} , \frac{2 - n + \sqrt{n^2 - 8n + 12}}{2} \right).
    \end{equation}
    We now focus on $k>1$, so $L(0) > 0$.
    Since $L(1) = -1$, the function has at least one zero on $(0,1)$; let $s_0$ be the first zero of $L$, and we will prove that it is the unique root.
    If $P(s)>0$ on $(0,1)$, then the computation $\on{sgn} L'(s_*) = - \on{sgn} P(s_*) < 0$ at a zero shows that $L$ cannot have an upward crossing, so it has exactly one zero $s_* = s_0$.
    Therefore, we focus on the case when $P(s)$ has exactly two zeros on $(0,1)$, which we denote by $s_1 < s_2$.
    We will prove $s_0 \geq s_1$.
    We further assume that $s_1 > \frac{k-1}{n-2}$, otherwise the result is clear from
    \[
    \tilde{L}_g(t) = t (1-t^2) \frac{g'_{\alpha}}{g_{\alpha}} + (k-1) - (n-2) t^2 > (k-1) - (n-2) t^2 > 0 \qquad \text{for } \; t \leq \sqrt{\tfrac{k-1}{n-2}}.
    \]
    Since $\hat{\alpha}<0$, the quadratic $P(s)$ has its absolute minimum over $\bR$ at $s = \frac{n-2k+ \hat{\alpha}}{2\hat{\alpha}}$.
    Since $P(0) , P(1) > 0$, the function has two roots $0 < s_1<s_2 < 1$ if $\frac{n-2k+\hat{\alpha}}{2 \hat{\alpha}} \in (0,1)$ and
    \[
    (n-2 k + \hat{\alpha})^2 + 4(k-1) \hat{\alpha} > 0 \iff |\hat{\alpha} + n-2| > 2\sqrt{(n-k-1) (k-1)} .
    \]
    We then compute the smaller root as
    \[
    s_1 = \frac{n - 2k + \hat{\alpha} + \sqrt{(n - 2k + \hat{\alpha})^2 + 4 (k-1) \hat{\alpha}}}{2 \hat{\alpha}}
    \]
    while $P( \frac{k-1}{n-2}) > 0$ implies $2(n-2) + \hat{\alpha}>0$.
    Consequently, $(n-2) + \frac{\hat{\alpha}}{k} > 0$ for $k \geq 2$.

    Since the ODE~\eqref{eqn:L-equation-this-one-promise} is first-order in $L$, we can prove that $L(s_1)>0$ by constructing an explicit subsolution $w(s)$ with $w(s_1) \geq 0$.
    Let $w(s) = (k-1) - \beta s$, for $\beta$ to be determined with $\beta > (n-2) + \frac{\hat{\alpha}}{k}$.
    Then, $w(0)= L(0)$ and $w'(0)<L'(0)$, so $w<L$ for small $s>0$.
    The ODE~\eqref{eqn:L-equation-this-one-promise} for $w(s)$ becomes
    \begin{align*}
        & 2s(1-s) w'(s) + w(s)^2 + (ns-k) w(s) + P(s) \\
        &\quad= s \left[ ( \beta^2 - (n-2) \beta- \hat{\alpha}) s + k(n-2-\beta) + \hat{\alpha} \right]
    \end{align*}
    whose sign is determined by the sign of the affine part.
    The negativity of this expression is therefore equivalent to negativity at $s=0$ and $s= s_1$.
    These conditions are respectively equivalent to
    \begin{equation}\label{eqn:beta-two-conditions}
    \beta \geq (n-2) + \tfrac{\hat{\alpha}}{k}, \qquad ( \beta^2 - (n-2) \beta - \hat{\alpha}) s_1 + k (n-2- \beta) + \hat{\alpha} < 0.
    \end{equation}
    We study the latter expression as a quadratic in $\beta$, so
    \[
    Q(\beta) = s_1 \beta^2 - ( (n-2) s_1 +k) \beta + (\hat{\alpha} - \hat{\alpha} s_1 + k (n-2)).
    \]
    Notably, $Q$ attains its minimum at $\beta_* = \frac{(n-2) s_1 + k}{2s_1}$, where
    \begin{align*}
        \min_{\beta} Q(\beta) &= Q(\beta_*) = - \frac{( (n-2) s_1 + k)^2}{4s_1} + \hat{\alpha}(1 - s_1) + k(n-2) \\
        &= n(k-1) - \frac{k-1}{s_1} - \frac{((n-2) s_1+ k)^2}{4s_1} \\
        &= \frac{4(k-1) (ns_1 - 1) - (k + (n-2) s_1)^2}{4s_1},
    \end{align*}
    where we used the fact that $s_1$ satisfies $P(s_1) = 0$, therefore
    \[
    \hat{\alpha}(1-s_1) = - (n-2k) - \tfrac{k-1}{s_1}.
    \]
    Finally, direct computation shows that
    \[
    4(k-1)(ns-1) - ( k + (n-2) s)^2 < 0 \qquad \text{for } \; s \in (0,1)
    \]
    holds for every $n \geq 3$ and $2 \leq k \leq n-2$.
    Indeed, the quadratic has values at the endpoints $s=0$ and $s=1$ given by $-4(k-1) - k^2 < 0$ and $-(n-k)^2 < 0$, respectively, while its discriminant is $\Delta = - 32(k-1) ( 2+n(n-k-2))< 0$.
    It follows that $\{ \beta : Q(\beta) < 0 \} \neq \varnothing$.

    We now verify that the set of $\beta$ satisfying both conditions~\eqref{eqn:beta-two-conditions} is non-empty.
    Since the first condition $\beta \geq (n-2) + \frac{\hat{\alpha}}{k}$ describes a ray, while $Q(\beta) < 0$ is a downward parabola, the two constraints have non-empty intersection provided that
    \[
    Q( n-2 + \tfrac{\hat{\alpha}}{k} ) < 0 \qquad \text{or} \qquad Q'(n-2 + \tfrac{\hat{\alpha}}{k}) < 0.
    \]
    The former condition is clear, whereas the latter would ensure that the ray $\beta \geq (n-2) + \frac{\hat{\alpha}}{k}$ has non-empty intersection with the set $\{ \beta : Q(\beta) \neq 0 \} \neq \varnothing$.
    We compute
    \begin{align*}
        Q( n-2 + \tfrac{\hat{\alpha}}{k}) &= k^{-2} \hat{\alpha} s_{1}( \hat{\alpha} +k(n-k-2)), \\
    Q'(\beta) &= 2 s_1 \beta - ( (n-2) s_1 + k), \qquad Q'( n-2 + \tfrac{\hat{\alpha}}{k}) = (n-2 + \tfrac{2 \hat{\alpha}}{k}) s_1 - k.
    \end{align*}
    The first simplification comes from the fact that $s_1$ is a root of $P (s)$, therefore satisfies
    \[
    - \hat{\alpha} s_1^2 + ( n-2k + \hat{\alpha}) s_1 + (k-1) = 0.
    \]
    Consequently, $Q( n - 2 + \frac{\hat{\alpha}}{k})< 0$ if $\hat{\alpha}>-k(n-k-2)$ already satisfies our condition.
    Otherwise, $\hat{\alpha} \leq - k(n-k-2)$ implies
    \begin{align*}
    Q' ( n-2 + \tfrac{ \hat{\alpha}}{k}) &= ( n- 2 + \tfrac{2 \hat{\alpha}}{k}) s_1 - k \leq (-n+2k+2) s_1 - k \\
    &< - n + 2k+2 - k = - (n-k-2) 
    \end{align*} 
    due to $s_1<1$, so $Q' ( n - 2 + \tfrac{\hat{\alpha}}{k}) <0$.
    It follows that in the above situation, we can always produce some $\beta > (n-2) + \frac{\hat{\alpha}}{k}$ such that the function $w(s) = (k-1) - \beta s$ is a strict subsolution of the ODE~\eqref{eqn:L-equation-this-one-promise} with $w(s_1) \geq 0$ and $w>0$ on $[0,s_1)$.
    Consequently, $L>0$ on $[0,s_1)$, whereby the first zero of $L$ satisfies $s_0 > s_1$.
    At that point, we have $L'(s_0) < 0$ and $\on{sgn} P(s_0) = - \on{sgn} L'(s_0) > 0$.
    Therefore, $s_0 \geq s_1$ and $P<0$ on $(s_1, s_2)$ forces $s_0 > s_2$.
    If $L$ had another zero $\tilde{s}_0 \in (s_0, 1)$, we would need $\on{sgn} L'(\tilde{s}_0) = - \on{sgn} P(\tilde{s}_0) \geq 0$, which is impossible since $P>0$ for $s>s_0$.
    This proves that $L$ has a unique sign change on $(0,1)$, and so does $\tilde{L}_g(t) = L(t^2)$.
    In particular, $\tilde{L}_g(t_*) > 0$ for some $t_* > 0$ implies that $\tilde{L}_g > 0$ on $[0,t_*]$, completing the proof.
\end{proof}

\begin{lemma}\label{lemma:interval-of-alpha}
Let $t_{n,k}>0$ denote the zero of the function $f_{n,k}(t) := {}_2 F_1 ( \frac{n-1}{2} , - \frac{1}{2} ; \frac{k}{2} ; t^2)$ and consider the following condition for $\alpha \in (2-n,0)$:
    \begin{equation}\label{eqn:g-n,k-inequality}
    \frac{g'_{\alpha}(t_{n,k})}{g_{\alpha}(t_{n,k})} > \frac{(n-2) t_{n,k} - (k-1) t_{n,k}^{-1}}{1 - t_{n,k}^2}.
    \end{equation}
    Let ${\cA}_{n,k} \subset (2-n,0)$ denote the set of $\alpha \in (2-n,0)$ for which the condition \eqref{eqn:g-n,k-inequality} holds; this is an interval with midpoint $1 - \frac{n}{2} = \frac{2-n}{2}$.
    For $\alpha \in \cA_{n,k}$, the function
    \[
        \frac{g'_{\alpha}(t)}{g_{\alpha}(t)} - \frac{ (n-2) t - (k-1) t^{-1} }{1-t^2}
    \]
    is everywhere positive on $(0,t_{n,k}]$.
\end{lemma}
\begin{proof}
    The hypergeometric function is symmetric in the first two arguments, meaning that
    \[
    g_{\alpha}(t) := {}_2 F_1 \left( \frac{n+\alpha-2}{2} , - \frac{\alpha}{2} ; \frac{k}{2} ; t^2 \right) = g_{2-n-\alpha}(t)
    \]
    is symmetric under the involution $\alpha \leftrightsquigarrow 2-n-\alpha$.
    The rational term $\frac{(n-2) t_{n,k} - (k-1) t_{n,k}^{-1}}{ 1 - t_{n,k}^2}$ is invariant under this symmetry, so inequality~\eqref{eqn:g-n,k-inequality} is invariant under this symmetry.
    The domain $\cA_{n,k}$ of viable $\alpha$ is thus symmetric with respect to $1 - \frac{n}{2}$.

    To see that $\cA_{n,k}$ is an interval, we prove that the map
    \begin{equation}\label{eqn:map-alpha-riccati}
        \alpha \mapsto \frac{g_{\alpha}'}{g_{\alpha}}, \qquad \alpha \in (2-n,0)
    \end{equation}
    is pointwise strictly increasing for $\alpha \in (2 - n , \frac{2-n}{2}]$ and pointwise strictly decreasing for $\alpha \in [\frac{2-n}{2}, 0)$.
    In particular, this will imply the corresponding monotonicity for $\alpha \mapsto \frac{g'_{\alpha}(t_{n,k})}{g_{\alpha}(t_{n,k})}$, hence the interval structure of $\cA_{n,k}$.
    Since $g$ satisfies~\eqref{eqn:laplace-rho-a-g}, writing $\frac{g''}{g} = v'_g + v_g^2$ shows that the function $v_g$ solves the Riccati ODE
    \begin{equation}\label{eqn:riccati-equation}
    (1-t^2) \left( \bigl(\tfrac{g'_{\alpha}}{g_{\alpha}} \bigr)' + \bigl( \tfrac{g'_{\alpha}}{g_{\alpha}} \bigr)^2 \right) + \left( (k-1) t^{-1} - (n-1) t \right) \tfrac{g'_{\alpha}}{g_{\alpha}} + \hat{\alpha} = 0, \qquad \hat{\alpha} := \alpha(\alpha+n-2).
    \end{equation}
    We let $\tfrac{g'_{\alpha}(t)}{g_{\alpha}(t)}= 2t v_{\hat{\alpha}}(t^2)$, so that $v_{\hat{\alpha}}(s)$ satisfies the ODE
    \begin{equation}\label{eqn:v-alpha-s-ODE}
        v'_{\hat{\alpha}} + v_{\hat{\alpha}}^2 + \frac{k - ns}{2s(1-s)} v_{\hat{\alpha}}+ \frac{\hat{\alpha}}{4s(1-s)} = 0.
    \end{equation}
    Using the power series expansion of the hypergeometric function, we observe that
    \[
    {}_2 F_1(a,b;c;s) = 1 + \tfrac{ab}{c} s + O (s^2), \quad \tfrac{d}{ds} {}_2 F_1(a,b;c;s) = \tfrac{ab}{c} + O(s), \quad \tfrac{d}{ds} \log {}_2 F_1(a,b;c;s) \big\rvert_{s=0} = \tfrac{ab}{c},
    \]
    and applying this discussion to our situation, for $g_{\alpha}$, we deduce that $v_{\hat{\alpha}}(0) = - \tfrac{\hat{\alpha}}{2k}$.
    The map $\alpha \mapsto \alpha (\alpha + n-2)$ is strictly decreasing for $\alpha \in (2-n, \tfrac{2-n}{2}]$ and strictly increasing for $\alpha \in [\frac{2-n}{2},0)$.
    Therefore, the monotonicity of~\eqref{eqn:map-alpha-riccati} is equivalent to the monotonicity of the map $\hat{\alpha} \mapsto v_{\hat{\alpha}}(t)$ for $\hat{\alpha} \in ( - ( \frac{n-2}{2})^2 , 0)$.
    Indeed, letting $w_{\hat{\alpha}} := \partial_{\hat{\alpha}} v_{\hat{\alpha}}$ denote the variation field along solutions produced by varying the parameter $\hat{\alpha}$ in~\eqref{eqn:v-alpha-s-ODE}, we equivalently want to show that $w_{\hat{\alpha}} < 0$ for $s \in [0,1]$.
    Differentiating~\eqref{eqn:v-alpha-s-ODE} in $\hat{\alpha}$, we find that
    \begin{align*}
    w'_{\hat{\alpha}} + \left( 2 v_{\hat{\alpha}}  + \tfrac{k-ns}{2 s(1-s)} \right) w_{\hat{\alpha}} &= - \tfrac{1}{4s(1-s)} , \\
    \tfrac{d}{ds} \bigl( \mu(s) w_{\hat{\alpha}}(s) \bigr) &= - \frac{1}{4s(1-s)} \mu(s), \qquad \mu(s) := \exp \left( \int_0^s \bigl( 2 v_{\hat{\alpha}}(\sigma) + \tfrac{k - n \sigma}{2 \sigma (1 - \sigma)} \bigr) \, d \sigma \right) 
\end{align*}
    whereby $\mu(s) w_{\hat{\alpha}}(s)$ is a strictly decreasing function.
    Moreover, $w_{\hat{\alpha}}(0) = \partial_{\hat{\alpha}} v_{\hat{\alpha}}(0) = - \frac{1}{2k} <0$, and since $\mu(s) > 0$, we conclude that $w_{\hat{\alpha}}(s) < 0$ for all $s \in [0,1]$, as claimed; this completes the proof.
    
    Finally, for $\alpha \in \cA_{n,k}$ satisfying~\eqref{eqn:g-n,k-inequality} with midpoint $1 - \frac{n}{2}$, the positivity of 
    \[
    \frac{g'_{\alpha}(t)}{g_{\alpha}(t)} - \frac{(n-2) t - (k-1) t^{-1}}{1-t^2} > 0 \qquad \text{for } \; t \in (0,t_{n,k}]
    \]
    follows from Lemma~\ref{lemma:riccati-g-is-positive}.

\end{proof}

\subsection{Estimates on the roots}

First we are interested in asymptotically sharp estimates on the roots $t_{n,k}$ of $f_{n,k}(t)$ for various ranges of $\frac{k}{n} \in (0,1)$.
In what follows, it is simpler to work with the zero $s_{n,k}$ of the function ${}_2F_1( \frac{n-1}{2}, - \frac{1}{2} ; \frac{k}{2} ; t^2)$, so that $t^2_{n,k} = s_{n,k}$.

\begin{lemma}\label{lemma:hypergeometric-zero-estimates}
    Let $s_{n,k} > 0$ denote the zero of the function ${}_2 F_1 (\frac{n-1}{2} , - \frac{1}{2} ; \frac{k}{2} ; s)$.
    For $n \geq 60$, we have
    \begin{align}
        s_{n,k} &\leq \tfrac{k}{n} + \tfrac{3}{5 \sqrt{n}}, & & \text{for } \; \tfrac{k}{n} \in [ \tfrac{1}{3}, \tfrac{7}{8} ], \label{eqn:sn,k-sqrt-bound-1} \\
        s_{n,k} & \leq \tfrac{k}{n} + \tfrac{2}{5 \sqrt{n}}, & & \text{for } \; \tfrac{k}{n} \in [ \tfrac{7}{8}, \tfrac{9}{10}], \label{eqn:sn,k-sqrt-bound-2} \\
        s_{n,k}&\leq \tfrac{k}{n} + \tfrac{9}{25 \sqrt{n}}, & & \text{for } \; \tfrac{k}{n} \in [ \tfrac{9}{10}, \tfrac{15}{16}].\label{eqn:sn,k-sqrt-bound-3}
    \end{align}
\end{lemma}
\begin{proof}
    We set $\lambda := \tfrac{k}{n}$, so the Lemma~is concerned with estimates of the form $s_{n,k} \leq \lambda + \frac{c}{\sqrt{n}}$ on three sub-intervals of $[ \tfrac{1}{3}, \frac{15}{16}]$.
    We develop a general framework for obtaining such estimates, and then specialize to the intervals at hand.
    The hypergeometric function satisfies
    \[
    2 s(1-s) F''(s) - n(s-\lambda) F'(s) + \tfrac{n}{2} F(s) + O(1) F(s) = 0
    \]
    and is strictly decreasing on $[0,1]$.
    In what follows, it suffices to consider $s_k > \frac{k}{n} = \lambda$, since the above bounds are immediate otherwise.
    We set $s = \lambda + \frac{z}{\sqrt{n}}$ and define $\phi(z) := F(\lambda + \frac{z}{\sqrt{n}})$, so that
    \begin{align*}
    F'(s) &= \sqrt{n} \, \phi'(z), \qquad F''(s) = n \, \phi''(z), \\
    2 s(1-s) &= 2 ( \lambda + O (n^{-1/2})) (1 - \lambda + O(n^{-1/2})) = 2 \lambda(1-\lambda) + O(n^{-1/2}),
    \end{align*}
    and $ns-k = n ( \lambda + \frac{z}{\sqrt{n}}) - n \lambda = \sqrt{n} z$.
    Dividing the transformed equation by $n$, we arrive at
    \[
    \bigl( 2\lambda(1-\lambda) + O(n^{-1/2}) \bigr) \phi''(z) - z \phi'(z) + \bigl( \tfrac{1}{2} + O(n^{-1}) \bigr) \phi(z) = 0.
    \]
    As $\lambda \in [ \frac{1}{3}, \frac{15}{16}]$ remains on a bounded interval away from the endpoints, we find that $\lambda(1-\lambda) > \frac{1}{50}$ remains uniformly bounded away from $0$.
    Therefore, the equation converges uniformly on compact sets (in the variable $z$) to the limit ODE
    \begin{equation}\label{eqn:phi-limit-ODE-n-to-infty}
    2 \lambda(1-\lambda) \phi''(z) - z \phi'(z) + \tfrac{1}{2} \phi(z) = 0.
\end{equation}
    Consequently, for large $n$ and $z$ on the scale $z = O(1)$ (meaning that $s-\lambda = O(n^{-1/2})$), the function $\phi(z)$ is well-approximated by the solution $\phi_{\lambda}$ of the ODE~\eqref{eqn:phi-limit-ODE-n-to-infty} with matching initial data at the turning point $s = \lambda$.
    Rescaling $z$ to $\tilde{z} := \frac{1}{\sqrt{2 \lambda(1-\lambda)}}z$, we arrive at the equation
    \[
    \phi'' - \tilde{z} \phi' + \tfrac{1}{2} \phi =0.
    \]
    The final equation is independent of $\lambda$ and has a first zero at a universal $\tilde{z}_0 > 0$; concretely, $\tilde{z}_0 \in [\frac{76}{100}, \frac{78}{100}]$.
    Returning to the variable $z$, we deduce that the rescaled function has a zero approaching $z_{\infty}(\lambda) = \sqrt{2\lambda(1-\lambda)} \, \tilde{z}_0$ as $n \to \infty$.
    Inspecting the above asymptotic bounds, we conclude that 
    \[
    s_{n,k} \leq \lambda + \frac{z_{\infty}(\lambda)}{\sqrt{n}} + \frac{C}{n \sqrt{n}} = \lambda + \frac{\sqrt{2 \lambda(1-\lambda)}}{\sqrt{n}} \tilde{z}_0 + \frac{C}{n \sqrt{n}}
    \]
    for an absolute constant $C$.
    A rough upper bound for the error term $C$ in this expansion is obtained by replacing each instance of $O(1), O(n^{-1/2})$, and $O(n^{-1})$ in the above estimates by the Taylor expansion with remainder; for $n \geq 60$, this enables us to bound $C < 5$.
    Moreover, direct computation by maximizing the quadratic $2 \lambda(1-\lambda)$ and using the exact value of $\tilde{z}_0$ implies that
    \[
    \sup_{[\frac{1}{3}, \frac{7}{8}]} \sqrt{2 \lambda(1-\lambda)} \, \tilde{z}_0 \leq  \frac{56}{100}, \qquad \sup_{[\frac{7}{8}, \frac{9}{10}]} \sqrt{2 \lambda(1-\lambda)} \, \tilde{z}_0 \leq  \frac{37}{100}, \qquad \sup_{[\frac{9}{10}, \frac{15}{16}]} \sqrt{2 \lambda(1-\lambda)} \, \tilde{z}_0 \leq  \frac{34}{100}.
    \]
    The conclusions~\eqref{eqn:sn,k-sqrt-bound-1}~--~\eqref{eqn:sn,k-sqrt-bound-3} are verified by direct computation for $60 \leq n \leq 300$.
    For $n \geq 300$, we apply the above bounds in the estimate for $s_{n,k}$ to obtain $s_{n,k} \leq \lambda + \frac{\tilde{c} + 5 n^{-1}}{\sqrt{n}} \leq \lambda + \frac{\tilde{c} + \frac{1}{60}}{\sqrt{n}}$.
    Finally, we observe that $c - \tilde{c} > \frac{1}{60}$ holds for the constants corresponding to each sub-interval; therefore, $s_{n,k} \leq \frac{k}{n} + \frac{c}{\sqrt{n}}$ completes the proof of~\eqref{eqn:sn,k-sqrt-bound-1}~--~\eqref{eqn:sn,k-sqrt-bound-3}.
\end{proof}

When $k$ is closer to $n$, we obtain an asymptotically sharp bound on $s_{n,k}$ as follows.
\begin{lemma}\label{lemma:oscillation-bound}
    Let $s_{n,k}>0$ denote the zero of the function ${}_2 F_1 ( \frac{n-1}{2}, - \frac{1}{2} ; \frac{k}{2} ; s)$ and suppose that $\frac{n}{2} \leq k \leq n-12$.
    Then, we have either $s_{n,k} < \frac{k}{n}$ or
    \begin{equation}\label{eqn:u-overshoot-bound}
    (n s_{n,k} - k)^2 \leq 2 n (1 - s_{n,k}).
    \end{equation}
\end{lemma}
\begin{proof}
For brevity, we only show the result when $d = n-k \geq 6$ is even.
The case when $d$ is odd can be proved by similar methods; we now make use of an identity for hypergeometric functions applicable when $\frac{d}{2} \in \bN^*$.
For brevity, we denote $s_k := s_{n,k}$ and $F_k(s) := {}_2 F_1( \frac{n-1}{2} , - \frac{1}{2} ; \frac{k}{2} ; s)$ and consider $s_{n,k} > \frac{k}{n}$, so the inequality~\eqref{eqn:u-overshoot-bound} can be rearranged into
\[
(ns_k - k)^2 < 2n(1 - s_k) \iff ( d - n (1 - s_k))^2 < 2n (1-s_k) \iff s_k < 1 - \frac{d+1 - \sqrt{2d+1}}{n}.
\]
The desired bound is therefore equivalent to proving $F_k(1-s) < 0$ for $s_* = \frac{d+1 - \sqrt{2d+1}}{n}$.
Applying Euler's transformation, we may write 
\[
{}_2F_1(a,b;c;s) = (1-s)^{c-a-b} {}_2F_1(c-a, c-b;c;s),
\]
which reduces the study of $F_k$ to that of $G(s) := {}_2 F{}_1 ( \frac{1-d}{2} , \frac{k+1}{2} ; \frac{k}{2} ; 1-s)$, since $c-a-b = 1 - \frac{d}{2}$ is a negative integer in this case.
Moreover, $s=0$ is a regular point of $G(s)$, with value given by Gauss' summation formula~\cite{dtmf}*{15.E.20} as
\[
G(0) = \frac{\Gamma( \frac{k}{2}) \Gamma( \frac{n-k}{2} - 1)}{\Gamma( \frac{n-1}{2} ) \Gamma( - \tfrac{1}{2}) } < 0
\]
where ${}_2F_1(a,b;c;s) = (1-s)^{\frac{2-d}{2}}G(1-s)$.
Observe that $\frac{k}{2} - ( \frac{1-d}{2} + \frac{k+1}{2}) = \frac{d}{2} \in \bN^*$, so $G(s)$ admits the expansion~\cite{dtmf}*{15.8.10} as a polynomial of degree $\frac{d}{2}-2$ plus an $O ( s^{\frac{d}{2}-1} \log s )$ remainder,
\allowdisplaybreaks{
\begin{align*}
     {}_2F{}_1&(a,b;a+b+m;1-s) \\
    &= \frac{1}{\Gamma(a+m)\Gamma(b+m)} \sum_{\ell=0}^{m-1} \frac{(a)_\ell(b)_\ell (m-\ell-1)!}{\ell!} (-1)^\ell s^\ell  \\
    &\quad - \frac{(-1)^m s^m}{\Gamma(a)\Gamma(b)} \sum_{\ell=0}^{\infty} \frac{(a+m)_\ell (b+m)_\ell}{\ell!(\ell+m)!} s^\ell \left[ \ln s - \psi(\ell+1) - \psi(\ell+m+1) + \psi(a+\ell+m) + \psi(b+\ell+m) \right].
\end{align*}}
Note that $m = \frac{d}{2}-1$ in our situation, while $\frac{1}{\Gamma(a) \Gamma(b)} = \frac{(a)_m (b)_m}{\Gamma(a+m) \Gamma(b+m) }$.
Since $\Gamma(a+m) = \Gamma ( \frac{1-d}{2} + \frac{d}{2} - 1) = \Gamma( - \frac{1}{2} ) < 0$ in our situation, see that the prefactor $\frac{1}{\Gamma(a+m) \Gamma(b+m)} = \frac{1}{\Gamma( - \frac{1}{2}) \Gamma ( \frac{n-1}{2}) } < 0$, so $G(s)$ has the opposite sign of $P_m(s) - s^m R_m(s)$, where $P_m(s)$ is a degree-$(\frac{d}{2}-2)$ polynomial and $R_m(s)$ is the $O( s^{\frac{d}{2}-1} \log s)$ remainder term.
Concretely,
\[
P_m(s) := \sum_{\ell=0}^{m-1} \frac{(a)_{\ell} (b)_{\ell} (m-\ell-1)!}{\ell!} (-1)^{\ell} s^{\ell}, \qquad R_m(s) :=(-1)^m (a)_m (b)_m \sum_{\ell=0}^{\infty} T_{\ell}(s) \Lambda_{\ell}(s)
\]
where we denote $T_{\ell}(s) := \frac{(a+m)_{\ell} (b+m)_{\ell}}{\ell! (\ell+m)!} s^{\ell}$ and
\begin{equation}\label{eqn:lambda-ell(s)}
\Lambda_{\ell}(s) := \log s - \psi(\ell+1) - \psi(m+\ell+1) + \psi( \ell - \tfrac{1}{2}) + \psi ( \tfrac{n-1}{2} + \ell)
\end{equation}
is composed of digamma function terms.
We observe that $(-1)^m (a)_m > 0$, therefore proving $G( \frac{d+1 - \sqrt{2d+1}}{n}) < 0$ will be equivalent to showing that
\[
P_m ( \tfrac{d+1 - \sqrt{2d+1}}{n}) - ( \tfrac{d+1 - \sqrt{2d+1}}{n})^m \, R_m ( \tfrac{d+1 - \sqrt{2d+1}}{n}) > 0.
\]
Working as in Lemma~\ref{lem:alpha-n,1}, and using $a = \frac{1-d}{2} < 0$, we see that each of the $m$ terms summed in the first expression is positive because $(a)_\ell = ( \frac{1-d}{2})_\ell$ alternates sign $(-1)^\ell$, which cancels out the factor $(-1)^\ell$.
Additionally, $(b)_{\ell} = ( \frac{k+1}{2})_{\ell} > 0$ because 
\[
\ell \leq m-1  = \tfrac{n-k}{2} - 2 < \tfrac{k+1}{2} \qquad \text{due to } \quad k > \tfrac{n}{2}.
\]
Therefore, $P_m(s) > 0$ for $s>0$, with a lower bound obtained from adding the first four terms.
The successive summands of $P_m(s)$ are compared via
\[
\frac{ \frac{1}{(\ell-1)!} (a)_{\ell-1}(b)_{\ell-1} (m-\ell)! (-1)^{\ell-1} s^{\ell-1}}{\frac{1}{\ell!} (a)_{\ell} (b)_{\ell}(m-\ell-1)! (-1)^{\ell} s^{\ell} } = \frac{(m-\ell) \ell}{s ( m-\ell +\frac{3}{2}) ( \frac{n-3}{2} -  m +\ell) }.
\]
Since $d = n - k \geq 12$, we have $m-1 = \frac{d}{2}-2 \geq 4$.
We may there apply the above property for $\ell \in \{m-2, m-3, m-4\}$, which (together with $\ell = m-1$) comprise the four top terms of $P_m(s)$, we find that the ratios are
\[
r_1 \geq \frac{m-1}{s \cdot \frac{5}{2} \cdot \frac{n}{2}} = \frac{2(d-4)}{5 ns}, \qquad r_2 \geq \frac{2(m-2)}{s \cdot \frac{7}{2} \cdot \frac{n}{2}} = \frac{4(d-6)}{7ns}, \qquad r_3 \geq \frac{3(m-3)}{s \cdot \frac{9}{2} \cdot \frac{n}{2}} = \frac{2(d-8)}{3ns}
\]
respectively.
Specializing this computation to $s_* = \frac{d+1 - \sqrt{2d+1}}{n}$, we find that 
\begin{align*}
r_1 &\geq \frac{2(d-4)}{5(d+1 - \sqrt{2d+1})} \geq \frac{2}{5}, \qquad r_2 \geq \frac{4(d-6)}{7 (d+1 - \sqrt{2d+1})} \geq \frac{4}{7} \cdot \frac{3}{4} \geq \frac{3}{7}, \\
r_3 &\geq \frac{2(d-8)}{3(d+1 - \sqrt{2d+1})} \geq \frac{2}{3} \cdot \frac{1}{2} = \frac{1}{3}
\end{align*}
because the functions $\frac{d-C}{d+1 - \sqrt{2d+1}}$ are strictly increasing for $d \geq 12$.
We conclude that
\begin{equation}\label{eqn:Pm(s*)-bound}
\begin{split}
P_m(s_*) &\geq \frac{(a)_{m-1}(b)_{m-1}}{(m-1)!} (-1)^{m-1} s_*^{m-1} ( 1+ r_1 + r_1 r_2 + r_1 r_2 r_3) \\
&> \frac{8}{5} \frac{ |(-m-\frac{1}{2})_{m-1}| ( \frac{k+1}{2})_{m-1} }{(m-1)!} s_*^{m-1}.
\end{split}
\end{equation}
Next, we consider the digamma function terms $\Lambda_{\ell}(s)$ of the series $R_m(s)$ from~\eqref{eqn:lambda-ell(s)}.
We will apply the bound $\log ( s- \tfrac{1}{2}) < \psi(s) < \log s$ for the digamma function, valid for all $s > \frac{1}{2}$, along with the special values
\[
\psi(1) = - \gamma, \qquad \psi( \tfrac{1}{2}) = - \gamma - 2 \log 2, \qquad \psi ( - \tfrac{1}{2}) = \psi( \tfrac{1}{2}) + 2.
\]
We combine $\psi ( \frac{n-1}{2}) \leq \log ( \frac{n-1}{2})$ with $\psi ( m+1) \geq \log ( m + \frac{1}{2}) $with the explicit constant $\psi( - \frac{1}{2}) - \psi(1) = 2 - 2 \log 2 \approx 0.6137$ to conclude that
\begin{align*}
    \Lambda_0(s_*) &=  \log s_* + \psi ( \tfrac{n-1}{2}) - \psi(m+1) + \bigl( \psi ( - \tfrac{1}{2}) - \psi(1) \bigr) \\
    &\leq \log s_* + \log ( \tfrac{n-1}{2}) - \log ( m + \tfrac{1}{2}) + \bigl( \psi ( - \tfrac{1}{2}) - \psi(1) \bigr) \\
    &\leq \psi ( - \tfrac{1}{2}) - \psi(1) = ( 2 - 2 \log 2) < \tfrac{62}{100}.
\end{align*}
In the last step, we used the fact that $m + \frac{1}{2} = \frac{d-1}{2}$ and $ns_* = d+1 - \sqrt{2d+1}$, whereby
\[
\log s_* + \log ( \tfrac{n-1}{2}) - \log ( m + \tfrac{1}{2}) < \log (ns_*) - \log ( d-1) < 0 
\]
because $d+1 - \sqrt{2d+1} < d-1$ due to $d>2$.
On the other hand, for $\ell\geq1$ we bound
\begin{align*}
    \Lambda_{\ell}(s_*) &\leq \log s_* + \log \Bigl( \frac{\frac{n-1}{2} + \ell}{\ell+m+ \frac{1}{2}} \Bigr) + \log \Bigl( \frac{\ell - \frac{1}{2}}{\ell + \frac{1}{2}} \Bigr) < \log \Bigl( s_* \frac{\frac{n-1}{2} + \ell}{\ell+m+ \frac{1}{2}} \Bigr) \\
    &< \log \Bigl( \frac{d+1}{n+1} \cdot \frac{\frac{n-1}{2}+1}{1 + \frac{d-1}{2}} \Bigr) = 0,
\end{align*}
which shows that $\Lambda_{\ell}(s_*) < 0$ for all $\ell \geq 1$.
Since the coefficients $T_{\ell}(s)$ of the series $R_m(s)$ are all positive, we conclude that
\[
R_m(s_*) < |(a)_m| (b)_m T_0(s_*) \Lambda_0(s_*) \leq |(a)_m| (b)_m \tfrac{1}{m!} \cdot \tfrac{62}{100}.
\]
Combined with the bound~\eqref{eqn:Pm(s*)-bound}, this implies that
\[
\frac{s_*^m \, R_m(s_*)}{P_m(s_*)} < \frac{62}{100} \cdot \frac{5}{8} \cdot s_* \cdot \frac{1}{m} \cdot \left| \frac{(a)_m (b)_m}{(a)_{m-1} (b)_{m-1}} \right| < \frac{31 s_*}{40 (d-2)} \cdot \frac{3n}{4}
\]
where we used $\frac{(a)_m}{(a)_{m-1}} = - \frac{3}{2}$ and $\frac{(b)_m}{(b)_{m-1}} = \frac{n-3}{2}$.
Consequently, 
\[
\frac{s_*^m \, R_m(s_*)}{P_m(s_*)} < \frac{93}{160} \frac{d+1 - \sqrt{2d+1}}{d-2} < \frac{9}{10},
\]
which implies that
\[
G(s_*) = \frac{1}{\Gamma( - \frac{1}{2} ) \Gamma ( \frac{n-1}{2})} \bigl( P_m(s_*) - s^m R_m(s_*) \bigr) < \frac{1}{\Gamma( - \frac{1}{2} ) \Gamma ( \frac{n-1}{2})} \bigl( P_m(s_*) - \tfrac{9}{10} P_m(s_*) \bigr) < 0.
\]
We conclude that $G(s_*) < 0$, hence $F_k(1-s_*) < 0$ and $s_k < 1 - \frac{d+1 - \sqrt{2d+1}}{n}$ completes the proof.
\end{proof}

\section{The admissible interval}\label{section:properties-of-subsolutions}

We now turn to the study of the interval $\cA_{n,k} \subset (2-n,0)$ of $\alpha$ satisfying the condition~\eqref{eqn:g-n,k-inequality}, as discussed in Lemma~\ref{lemma:interval-of-alpha}.
The fact that $\cA_{n,k} \neq \varnothing$ does not follow from this definition, and is a highly non-trivial property: in fact, $\cA_{n,k} = \varnothing$ for $n \leq 6$, corresponding to the instability of the cones $U_{n,k}$ in low dimensions.
We will now show that $4 - n \in \cA_{n,k}$ for all $n \geq 7$ and $1 \leq k \leq n-2$.

\subsection{Near-endpoint \texorpdfstring{$k$}{k}}

For technical reasons, we first treat the ``edge'' cases $k \in \{1,2,n-3, n-2\}$.

\begin{lemma}\label{lem:alpha-n,1}
    For every $n \geq 7$ and $k \in \{ 1 , 2 , n-3, n - 2 \}$, we have
    \[
    4 - n \in \cA_{n,k}
    \]
    in the notation of Lemma~\ref{lemma:riccati-g-is-positive}.
\end{lemma}
\begin{proof}
    For $\alpha = 4-n$, we write $\tilde{g}_{n,k}(t) := g_{n,k,4-n}(t) := {}_2 F_1 \Bigl( 1 , \frac{n-4}{2} ; \frac{k}{2} ; t^2 \Bigr)$.
    We will treat each $k \in \{ 1 , 2 , n-3, n-2 \}$ separately.

\noindent\textbf{Case I: $k=1$}.
We first compute directly that
\[
    \frac{ \tilde{g}'_{n,1}(t_{n,1}) }{ \tilde{g}_{n,1}(t_{n,1}) } - \frac{(n-2) t_{n,1}}{1 - t_{n,1}^2} \geq \frac{ \tilde{g}'_{7,1}(t_{7,1}) }{ \tilde{g}_{7,1}(t_{7,1}) } - \frac{5 t_{7,1}}{1 - t_{7,1}^2} > 5 \cdot 10^{-2}
\]
    for $7 \leq n \leq 10$, so we focus on $n \geq 10$ in what follows.
    We introduce the function
    \[
    \phi(s) := {}_2 F_1 \left( - \frac{1}{2} , \frac{5-n}{2} ; \frac{1}{2} ; s \right), \qquad s := t^2,
    \]
    whereby a standard duality transformation for hypergeometric functions gives
    \[
    g_{n,1,4-n}(t) = (1-t^2)^{- \frac{n-3}{2}} \phi(t^2).
    \]
    We therefore obtain
    \begin{align*}
    \frac{\tilde{g}'_{n,1}(t)}{\tilde{g}_{n,1}(t)} - \frac{(n-2) t}{1-t^2} &= \frac{(n-3) t}{1-t^2} + 2 t \frac{\phi'(t^2)}{\phi(t^2)} - \frac{(n-2) t}{1-t^2} = \frac{t}{1-t^2} \left( 2 (1-t^2) \frac{\phi'(t^2)}{\phi(t^2)} - 1 \right)
    \end{align*}
    and the desired positivity is equivalent to
    \[
    \frac{\phi'(t_0^2)}{\phi(t_0^2)} > \frac{1}{2 (1-t_0^2)}, \qquad \text{where } \; t_0 := t_{n,1}.
    \]
    We also recall that the zero $t_{n,1}$ of $f_{n,1}(t)$ satisfies $t_{n,1} < \sqrt{\frac{2}{n-1}}$, since $f_{n,1}(t) < 1 - \frac{n-1}{2}t^2$ due to the bound~\eqref{eqn:Fk(s)-cubic-bound}.
    To prove the claim, it therefore suffices to show that
    \[
    E(s) := 2(1-s) \phi'(s) - \phi(s) > 0 \qquad \text{for } \; s \in (0, \tfrac{2}{n-1}], \qquad n \geq 10.
    \]
    The hypergeometric function $\phi(s)$ has the power series expansion
    \[
    \phi(s) = \sum_{m=0}^{\infty} a_m s^m, \qquad a_m = \frac{(- \frac{1}{2})_m ( \frac{5-n}{2})_m}{( \frac{1}{2})_m m!}
    \]
    where $(\alpha)_m := \frac{\Gamma(\alpha+m)}{\Gamma(\alpha)} = \alpha(\alpha+1) \cdots (\alpha+m-1)$.
    We compute
    \[
    a_0 = 1, \qquad a_1 = \frac{n-5}{2}, \qquad a_2 = - \frac{(n-5)(n-7)}{24}, \qquad a_3 = \frac{(n-5)(n-7)(n-9)}{240},
    \]
    Applying the argument of Lemma~\ref{lemma:riccati-g-is-positive}, we may reduce the positivity of $2(1-s) \phi'(s) - \phi(s)$ for $s \in (0, \frac{2}{n-1}]$ to the computation of
    \[
    E( \tfrac{2}{n-1}) := 2 ( 1 - \tfrac{2}{n-1}) \phi'(\tfrac{2}{n-1}) - \phi(\tfrac{2}{n-1}) > 0.
    \]
    Using the above power series expansion, we obtain
    \begin{align*}
        E(s)&= 2 (1-s) \phi'(s) - \phi(s) = \sum_{m=0}^{\infty} E_m s^m \\
        &\geq (n-6) - \frac{(n-5)(n+2)}{6} s + \frac{(n-5)(n-7)}{120} \left( (3n-2) + \frac{(6-5n)(n-9)}{14} s \right) s^2 - R_4 
    \end{align*}
    for a remainder term $R_4 = \sum_{m=4}^{\infty} E_m s^m$, involving terms of order $4$ and higher, which we now estimate for $s \in (0, \frac{2}{n-1}]$ and $n \geq 10$.
    Note that 
    \[
    E_m = 2 (m+1) a_{m+1} - (2m+1) a_m.
    \]
    For each $m \leq \frac{n-5}{2}$, we have
    \[
    \left| ( \tfrac{5-n}{2})_m \right| = \left| \left( \tfrac{5-n}{2} - 1 \right) \cdots \left( \tfrac{5-n}{2} - m + 1 \right) \right| \leq ( \tfrac{n-5}{2} )^m
    \]
    as well as
    \[
    \left| \frac{(- \frac{1}{2})_m}{( \frac{1}{2})_m} \right| = \prod_{j=0}^{m-1} \frac{|j- \frac{1}{2}|}{j + \frac{1}{2}} \leq 1, \qquad \text{for every } \; m \geq 0,
    \]
    because all factors after $j=0$ are $j <1$.
    We therefore obtain
    \[
    |a_m| = \left| \frac{(- \frac{1}{2})_m ( \frac{5-n}{2})_m}{(\frac{1}{2})_m m!} \right| \leq \frac{1}{m!} \left| \Bigl( \frac{5-n}{2}\Bigr)_m \right|
    \]
    and, in particular, $|a_m| \leq \frac{1}{m!} ( \frac{n-5}{2})^m $ for $m \leq \frac{n-5}{2}$.
    Therefore, $s \in (0, \frac{2}{n-1}]$ gives 
    \begin{align*}
    |a_m| s^m &\leq \frac{1}{m!} \Bigl( \frac{n-5}{2} \Bigr)^m \Bigl( \frac{2}{n-1} \Bigr)^m = \frac{1}{m!} \Bigl( \frac{n-5}{n-1} \Bigr)^m, \\
    |E_m| s^m &\leq 2 (m+1) |a_{m+1}| s^m + (2m+1) |a_m| s^m \\
    &\leq 2 \frac{m+1}{(m+1)!} \cdot \frac{n-1}{2} + \frac{2m+1}{m!} = \frac{1}{m!} (n-1) + \frac{2}{(m-1)!} + \frac{1}{m!}, \\
    \sum_{m=4}^{\frac{n-5}{2}} |E_m| s^m &< (n-1) \sum_{m=4}^{\infty} \frac{1}{m!} + 3 \sum_{m=3}^{\infty} \frac{1}{m!} = \bigl( e - \tfrac{8}{3} \bigr) (n-1)  + 3 \bigl( e - \tfrac{5}{2} \bigr) \\
    &< \tfrac{6}{105} n + \tfrac{9}{10}.
    \end{align*}
    We now suppose that $n$ is even, otherwise $\frac{5-n}{2} = \frac{5-(2\ell+1)}{2} = 2 - \ell \in \bN$ would imply that $\phi(s)$ is a polynomial of degree $\ell-2$, with $E_m = 0$ for $m>\ell-2$.
    We write $n = 2 ( \ell+2)$, so that
    \begin{align*}
    \Gamma ( \tfrac{1}{2} - \ell ) &= \frac{(-4)^{\ell} \ell!}{(2 \ell)!} \sqrt{\pi}, \qquad \left( \frac{5-n}{2} \right)_m = \frac{\Gamma ( \frac{1}{2} - \ell+m)}{\Gamma ( \frac{1}{2} - \ell)} \leq \frac{(m-1)!}{\Gamma(\frac{1}{2} - \ell)}, \\
    |a_m| &\leq \frac{1}{|\Gamma( \frac{1}{2} - \ell)|} \cdot \frac{(m-1)!}{m!} = \frac{1}{m \cdot |\Gamma ( \frac{1}{2} - \ell)|}, \\
    \sum_{m=\ell+1}^{\infty} |E_m| s^m &\leq \sum_{m = \ell+1}^{\infty} (2 (m+1) |a_{m+1}| + (2m+1) |a_m|) s^m \\
    &\leq \frac{4}{|\Gamma ( \frac{1}{2} - \ell)|} \sum_{m = \ell+1}^{\infty} s^m = \frac{4}{|\Gamma( \frac{1}{2} - \ell)|} \cdot \frac{s^{\ell+1}}{1-s}.
    \end{align*}
    Using Stirling's factorial formula for $\ell \geq 5$, we have 
    \[
    \frac{1}{2} \Bigl( \frac{\ell}{e} \Bigr)^{\ell} \cdot \sqrt{2/\pi} < \frac{(2 \ell)!}{4^{\ell} \ell!} < 2 \Bigl( \frac{\ell}{e} \Bigr)^{\ell} \cdot \sqrt{2/\pi}.
    \]
    Finally, we observe that $s \leq \frac{2}{n-1}$ has
    \[
    s^{\ell+1} \leq \left( \frac{2}{2 \ell + 3} \right)^{\ell+1} = \left( \frac{2 \ell}{2 \ell+3} \cdot \frac{1}{\ell} \right)^{\ell} \frac{1}{2 \ell+3} \leq e^{-\ell}
    \]
    again by Stirling's formula.
    We may therefore estimate the tail term by
    \[
    \sum_{m = \ell+1}^{\infty} |E_m| s^m \leq \frac{1}{1-s} \sqrt{\tfrac{2}{\pi}} \cdot e^{-\ell} \leq 5 e^{- \frac{n-4}{2}} \qquad \text{for all } \; n \geq 10.
    \]
    Combining the two estimates, we arrive at
    \[
    |R_4| \leq \sum_{m=4}^{\frac{n-5}{2}} |E_m| s^m + \sum_{\frac{n-5}{2}}^{\infty} |E_m| s^m < \frac{6}{105} n + \frac{9}{10} + \frac{1}{10} = \frac{6n + 105}{105}.
    \]
    Evaluating the above expression for $E(s)$ at $s = \frac{2}{n-1}$, we finally obtain
    \begin{align*}
    E \Bigl( \frac{2}{n-1} \Bigr) &> \frac{2(n-7) ( 39 n^3 - 156 n^2 + 145 n - 20) }{105 (n-1)^3} - \frac{6 (n+18)}{105} \\
    &> \frac{ 12 (n-1)^3 (6n-61) + 8 (31 n^2 - 43) }{105 (n-1)^3}
    \end{align*}
    which is positive for $n \geq 10$.
    Therefore, $E(s)>0$ for $s \in (0, \frac{2}{n-1}]$, completing the proof for $k=1$.

    \smallskip
    \noindent\textbf{Case II: $k=2$.} 
    The hypergeometric expression for $g(t)$ simplifies to
    \[
    g_{n,2,4-n}(t) = {}_2 F_1 \bigl( 1, \tfrac{n-4}{2} ; 1 ; t^2 \bigr) = (1-t^2)^{- \frac{n-4}{2}}, \qquad \frac{\tilde{g}'_{n,2}(t)}{\tilde{g}_{n,2}(t)} = (n-4) \frac{t}{1-t^2}.
    \]
    The desired inequality therefore becomes
    \[
    \frac{\tilde{g}'_{n,2}(t_{n,2})}{\tilde{g}_{n,2}(t_{n,2})} - \frac{(n-2) t_{n,2} - t_{n,2}^{-1}}{1 - t_{n,2}^2} > 0 \iff (n-4) t_{n,2} > (n-2) t_{n,2} - t_{n,2}^{-1}
    \]
    which is equivalent to $t_{n,2}^2 < \tfrac{1}{2}$ for $n \geq 10$.
    We recall that the zero $t_{n,2}$ of $f_{n,2}(t)$ satisfies $t_{n,2} < \sqrt{\frac{4}{n-1}}$, since $f_{n,2}(t) < 1 - \frac{n-1}{4}t^2$ due to the bound~\eqref{eqn:Fk(s)-cubic-bound}.
    For $n \geq 10$, this implies $t_{n,2} < \frac{2}{3} < \frac{1}{\sqrt{2}}$, from which the claimed positivity follows.

    \smallskip
    \noindent{\textbf{Case III: $k=n-2$.}}
    In this case $\tilde{g}_{n,n-2}(t) = {}_2 F_1 ( 1, \frac{n-4}{2} ; \frac{n-2}{2} ; t^2)$, for which the integral representation~\eqref{eqn:hypergeometric-integral-formula} becomes
    \begin{align*}
    g(t) &= \frac{\Gamma ( \frac{n-2}{2}) }{\Gamma ( \frac{n-4}{2}) \Gamma(1)} \int_0^1 u^{\frac{n-6}{2}} \frac{d u}{1 - t^2 u} = \frac{n-4}{2} \int_0^1 \frac{u^{\frac{n-6}{2}}}{1 - t^2 u} \, du, \\
    g'(t) &= (n-4) t \int_0^1 \frac{u^{\frac{n-6}{2}+1}}{(1 - t^2 u)^2} \, du.
    \end{align*}
    The quantity of interest then becomes
    \begin{align*}
    & \frac{\tilde{g}'_{n,n-2}(t_{n,n-2})}{\tilde{g}_{n,n-2}( t_{n,n-2})} - \frac{(n-2) t_{n,n-2} - (n-3) t_{n,n-2}^{-1}}{1 - t_{n,n-2}^2} \\
    &\quad= 2 t_{n,n-2} \left( \int_0^1 \frac{u^{\frac{n-6}{2}}}{1 - t_{n,n-2}^2 u} \, du \right)^{-1} \left( \int_0^1 \frac{ u^{\frac{n-6}{2} + 1} }{(1 - t_{n,n-2}^2 u)^2} \, du \right) - \frac{(n-2) t_{n,n-2} - (n-3) t_{n,n-2}^{-1}}{1 - t_{n,n-2}^2}.
    \end{align*}
    Let us set $\rho := n (1 - t_{n,n-2}^2)$, with $\rho > \frac{1}{4}$ for $n \geq 10$; equivalently, $t^2_{n,n-2} = 1 - \frac{\rho}{n}$ and $f_{n,n-2} ( \sqrt{1 - \frac{1}{4n} } ) < 0$.
    Using the expansion~\cite{dtmf}*{15.4.10} for hypergeometric functions of the form ${}_2 F_1 (a,b;a+b;s)$, valid here due to $a=\frac{n-1}{2}, b=- \frac{1}{2}, c = \frac{n-2}{2}$, we obtain
    \begin{align*}
    & {}_2 F_1 \bigl( \tfrac{n-1}{2}, - \tfrac{1}{2} ; \tfrac{n-2}{2} ; s) \\
    &\quad= \frac{\Gamma ( \frac{n-2}{2})}{\Gamma( \frac{n-1}{2}) \Gamma( -\frac{1}{2})} \left( - \ln(1-s) + \psi ( \tfrac{n-1}{2} ) + \psi ( - \tfrac{1}{2}) - 2 \, \psi(1) \right) + O( (1-s) |\log(1-s)| )
    \end{align*}
    for $\psi(s)$ the digamma function.
    We note that $\frac{\Gamma ( \frac{n-2}{2})}{\Gamma( \frac{n-1}{2}) \Gamma( -\frac{1}{2})} = - \frac{\Gamma ( \frac{n-2}{2})}{2 \sqrt{\pi}\Gamma( \frac{n-1}{2})} < 0$ while $t_{n,n-2}^2 >t_{n,n-3}^2 = \frac{n-3}{n-2}$, whereby ${}_2 F_1 ( \frac{n-1}{2}, - \frac{1}{2}; \frac{n-2}{2}, 1 - \frac{4}{n})$ has the opposite sign of
    \[
    - \log \tfrac{n}{4} + \psi ( \tfrac{n-1}{2} ) + \psi ( - \tfrac{1}{2}) - 2 \, \psi(1).
    \]
    Using the digamma function bound $\psi( \frac{n-1}{2}) \geq \log \frac{n-1}{2} - \frac{1}{n-1}$ for $n \geq 10$, and the explicit values of 
    \[
\psi(1) = - \gamma, \qquad \psi( \tfrac{1}{2}) = - \gamma - 2 \log 2, \qquad \psi ( - \tfrac{1}{2}) = \psi( \tfrac{1}{2}) + 2,
\]
we verify that this final quantity is positive for $n \geq 10$, whereby $f_{n,n-2} ( \sqrt{1-\frac{1}{4n}}) < 0$.
    Therefore, $\rho > \frac{1}{4}$.
    We perform the change of variables $u = 1 - \frac{\tau}{n}$ in the above integrals to write
    \begin{align*}
    u^{\frac{n-6}{2}} &= (1 - \tfrac{\tau}{n})^{\frac{n-6}{2}} =\exp \Bigl( \tfrac{n-6}{2} \log (1 - \tfrac{\tau}{n}) \Bigr) = \exp \Bigl( - \tfrac{\tau}{2} + O ( \tfrac{\tau^2}{n}) \Bigr), \\
    1 - t_{n,n-2}^2 u &= 1 - \bigl( 1 - \tfrac{\rho}{n} \bigr) (1 - \tfrac{\tau}{n} \bigr) = \tfrac{\tau + \rho}{n} + O(n^{-2})
    \end{align*}
    uniformly in $n \geq 10$.
    Using these substitutions, we may bound the quantity by
    \begin{align*}
    & \frac{\tilde{g}'_{n,n-2}(t_{n,n-2})}{\tilde{g}_{n,n-2}( t_{n,n-2})} - \frac{(n-2) t_{n,n-2} - (n-3) t_{n,n-2}^{-1}}{1 - t_{n,n-2}^2} \\
    &\quad \geq 2n \left(  \left( \int_0^{\infty} \frac{e^{- \tau/2}}{\tau + \rho} \, d \tau \right)^{-1} \left( \int_0^{\infty} \frac{e^{-\tau/2}}{(\tau+\rho)^2} \, d \tau \right) - \frac{1-\rho}{2\rho} \right) - 5.
    \end{align*}
    In the last step, we absorbed the error terms $O(\frac{\tau^2}{n})$ and $O(n^{-2})$ at each step into the conservative upper bound $5$.
    This rough upper bound is obtained by replacing each instance of $O(\frac{\tau^2}{n})$ and $O(n^{-2})$ in the above estimates by the Taylor expansion with remainder; for $n \geq 20$, this enables us to bound $C < 5$.
    For $\rho > \frac{1}{4}$, we find that
    \[
    \left( \int_0^{\infty} \frac{e^{- \tau/2}}{\tau + \rho} \, d \tau \right)^{-1} \left( \int_0^{\infty} \frac{e^{-\tau/2}}{(\tau+\rho)^2} \, d \tau \right) - \frac{1-\rho}{2\rho} > \frac{3}{20},
    \]
    which proves the desired positivity for $n \geq 20$.
    Finally, the inequality is checked to hold in the range $n \in \{ 10, 11, \dots, 20 \}$ by direct evaluation.

    \smallskip
    \noindent\textbf{Case IV: $k=n-3$.}
    The hypergeometric function $f_{n,n-3}(t)$ has the explicit expression $f_{n,n-3}(t) = \frac{(n-3) - (n-2) t^2}{(n-3) \sqrt{1-t^2}}$, so $t_{n,n-3} = \sqrt{\frac{n-3}{n-2}}$.
    The desired inequality becomes
    \[
    \frac{\tilde{g}'_{n,n-3}}{\tilde{g}_{n,n-3}} \bigl( \sqrt{\tfrac{n-3}{n-2}} \bigr) > (n-2) \sqrt{\tfrac{n-2}{n-3}}.
    \]
    We again represent $\tilde{g}_{n,n-3}(t) = {}_2 F_1 ( 1, \frac{n-4}{2} ; \frac{n-3}{2} ; t^2)$ via~\eqref{eqn:hypergeometric-integral-formula} to obtain
    \begin{align*}
        g(t) &= \frac{\Gamma ( \frac{n-3}{2}) }{\sqrt{\pi}\Gamma ( \frac{n-4}{2})} \int_0^1 u^{\frac{n-6}{2}} ( 1-u)^{- \frac{1}{2}} ( 1 - t^2 u)^{-1} \, du, \\
        g'(t) &= 2t \frac{\Gamma ( \frac{n-3}{2}) }{\sqrt{\pi}\Gamma ( \frac{n-4}{2})} \int_0^1 u^{\frac{n-4}{2}} ( 1-u)^{- \frac{1}{2}} ( 1- t^2 u)^{-2} \, du .
    \end{align*}
    We will estimate $\frac{\tilde{g}'_{n,n-3}}{\tilde{g}_{n,n-3}}$ as in Case III.
    We observe that $t_{n,n-3} = \sqrt{\frac{n-3}{n-2}}$, so $\rho = n (1 - t_{n,n-3}^2) = \frac{n}{n-2}$ in this case.
    Using the same transformation $u = 1 - \frac{\tau}{n}$ as above, we may write
    \[
    u^{\frac{n-6}{2}} = \exp \bigl( - \tfrac{\tau}{2} + O ( \tfrac{\tau^2}{n}) \bigr), \qquad 1 - t^2_{n,n-3} u = \tfrac{\tau + \rho}{n} + O (n^{-2}), \qquad (1-u)^{-\frac{1}{2}} =  \sqrt{n}
    \tau^{-\frac{1}{2}} \, .
    \]
    Absorbing the small error terms as in Case III above, now with $\tau + \rho = \tau + \frac{n}{n-2}$, we estimate
    \[
    \frac{\tilde{g}'_{n,n-3}}{\tilde{g}_{n,n-3}} \bigl( \sqrt{\tfrac{n-3}{n-2}} \bigr) > 2 \left( \int_0^{\infty} \frac{e^{- \tau/2} \tau^{- \frac{1}{2}}}{ \tau + \frac{n}{n-2}} \, d \tau \right)^{-1} \left( \int_0^{\infty} \frac{e^{- \tau/2} \tau^{- \frac{1}{2}}}{(  \tau + \frac{n}{n-2})^2} \, d \tau \right) \cdot n \cdot \sqrt{\tfrac{n-3}{n-2}} - 10
    \]
    by estimating as above and noting that $\rho = n (1 - t_{n,n-3}^2) = n (1 - \frac{n-3}{n-2}) = \frac{n}{n-2}$ in this case.
    The lower bound for $\frac{\tilde{g}'_{n,n-3}}{\tilde{g}_{n,n-3}} \bigl( \sqrt{\tfrac{n-3}{n-2}} \bigr)$ therefore becomes
    \begin{align*}
        &   2 n \sqrt{\tfrac{n-3}{n-2}} \left( \int_0^{\infty} \frac{e^{- \tau/2} }{\sqrt{\tau} (\tau+ 1 + \frac{2}{n-2})} \, d \tau \right)^{-1} \left( \int_0^{\infty} \frac{e^{- \tau/2}}{\sqrt{\tau}( \tau+ 1 + \frac{2}{n-2})^2} \, d \tau \right) - 10 > n \sqrt{\tfrac{n-3}{n-2}} \cdot \tfrac{7}{5} - 10
    \end{align*}
    for $n \geq 33$, because the function $s \mapsto 2 \left(\int_0^{\infty} \frac{e^{- \tau/2} }{\sqrt{\tau} (\tau+ 1 + s)} \, d \tau \right)^{-1}\left( \int_0^{\infty} \frac{e^{- \tau/2}}{\sqrt{\tau}( \tau+ 1 + s)^2} \, d \tau \right)$ is strictly decreasing in $s>0$, and larger than $\frac{7}{5}$ for $s = \frac{2}{n-2} \leq \frac{1}{10}$ when $n\geq 33$.
    We conclude that 
    \begin{align*}
    \frac{\tilde{g}'_{n,n-3}}{\tilde{g}_{n,n-3}} \bigl( \sqrt{\tfrac{n-3}{n-2}} \bigr)  - (n-2) \sqrt{\tfrac{n-2}{n-3}} &> \tfrac{7}{5} n \sqrt{\tfrac{n-3}{n-2}} - (n-2) \sqrt{\tfrac{n-2}{n-3}} - 10 \\
    &> \tfrac{2}{5} n \sqrt{\tfrac{n-3}{n-2}} - 10 >  0
    \end{align*}
    for $n \geq 33$.
    We directly evaluate that the inequality~\eqref{eqn:g-n,k-inequality} required for $4-n \in \cA_{n,n-3}$ holds in the range $n \in \{ 10, 11, \dots, 33 \}$; this completes the proof.
\end{proof}

We now move on to establishing that $4 - n \in \cA_{n,k}$ for all $n \geq 7$ and $3 \leq k \leq n-4$.
The result will be proved by considering different regimes for $k$: ``small'' $k$, with $1 \leq k \leq \frac{n - \sqrt{8(n-4)}}{2}$; ``intermediate'' $k$, with $\tfrac{1}{3} n \leq k \leq \tfrac{15}{16} n$; ``large'' $k$, with $\tfrac{15}{16} n \leq k \leq n-12$; and a final set of $n-11 \leq k \leq n-2$.
Each case is treated using different techniques, with a unified starting point.
Working in the framework of Lemma~\ref{lemma:riccati-g-is-positive}, we consider the function
    \[
    L_k(s) := \tilde{L}_g(t) \quad \text{for } \; s := t^2, \qquad \tilde{L}_g(t) := t (1-t^2) \frac{g'_{\alpha}(t)}{g_{\alpha}(t)} - (n-2) t^2 + (k-1),
    \]
    which satisfies the ODE~\eqref{eqn:L-equation-this-one-promise},
    \[
    2 s (1-s) L'(s) + L(s)^2 + (ns - k) L(s) + P_{n,k}(s) = 0,
    \]
    with $\hat{\alpha} = 2(4-n)$ in the notation of Lemma~\ref{lemma:riccati-g-is-positive}.
    This means that
    \[
    P_{n,k}(s) = (k-1) + (8 - n- 2k)s + 2(n-4) s^2, \qquad
    L(0) = k-1, \quad L'(0) = - (n-2) + \tfrac{2(n-4)}{k},
    \]
    and we recall that $L_k(1) = -1 $ for all $k$.
    In this notation, the condition~\eqref{eqn:g-n,k-inequality} for $4 - n \in \cA_{n,k}$ assumes the form $L(s_{n,k}) > 0$, where $s_{n,k}$ denotes the zero of the function ${}_2 F_1 ( \frac{n-1}{2} , - \frac{1}{2}; \frac{k}{2} ; s)$.

\subsection{Intermediate \texorpdfstring{$k$}{k}}

We address the range $3 \leq k \leq \frac{15}{16} n$.

\begin{lemma}\label{lemma:4-n.pt.1}
    For every $n \geq 7$ and $3 \leq k \leq \frac{n - \sqrt{8(n-4)}}{2}$, we have $4-n \in \cA_{n,k}$.
\end{lemma}
\begin{proof}
    For $7 \leq n \leq 35$, the result $L_k(s_k) > 0$ is verified by direct computation; in fact, we find that
    \[
    \min_{1 \leq k \leq n-2} L_{n,k} (s_{n,k}) \geq \min_{1 \leq k \leq 5} L_{7,k}(s_{7,k}) > 3 \cdot 10^{-2} > 0,
    \]
    so we may focus on $n \ge 35$ in what follows.
    Since $4-n \in \cA_{n,1} \cap \cA_{n,2}$ was proved in Lemma~\ref{lem:alpha-n,1}, we may further restrict ourselves to $k \geq 3$ for the subsequent arguments.
    Our argument in Lemma~\ref{lemma:riccati-g-is-positive} showed that $P_{n,k}(s)$ has two roots $0 < \hat{s}_1 < \hat{s}_2 < 1$ if and only if
    \[
    n-6 > 2 \sqrt{(n-k-1) (k-1)} \iff |k - \tfrac{n}{2}| > \sqrt{2n-8},
    \]
    in which case $L_k(\hat{s}_2) > 0$, meaning that $L_k$ remains positive on $[0,\hat{s}_2]$.
    The latter inequality is equivalent to $k < \frac{n - \sqrt{8(n-4)}}{2}$, and the larger root $\hat{s}_2$ of $P_{n,k}$ is given by
    \[
    \hat{s}_2 = \frac{n+2k-8 + \sqrt{(n+2k-8)^2 - 8(n-4)(k-1)}}{4(n-4)}.
    \]
    We recall that the function $F(s)$ satisfies the cubic bound~\eqref{eqn:Fk(s)-cubic-bound}, hence $F(s) \leq 1 - \frac{n-1}{2k}s$ and $s_{n,k} \leq \tfrac{2k}{n-1}$.
    We will show that $L_k(s_k) > 0$ by proving that $s_{n,k} \leq \hat{s}_2$.
    Since $\frac{2k}{n-1} > \hat{s}_1 = \frac{n+2k-8 - \sqrt{(n+2k-8)^2 - 8(n-4)(k-1)}}{4(n-4)}$, showing that $s_{n,k} \leq \hat{s}_2$ becomes equivalent to $P_{n,k}( \frac{2k}{n-1}) \leq 0$.
    We compute this expression as
    \[
    P_{n,k}( \tfrac{2k}{n-1}) = \tfrac{1}{(n-1)^2} \left[ (k-1)(n-1)^2 + 2k (n-1) (8-n-2k) + 8(n-4) k^2 \right] 
    \]
    and the condition $P_{n,k} ( \frac{2k}{n-1}) \leq 0$ becomes a quadratic in $k$, with solution interval
    \[
    3 \leq k \leq \frac{(n-1) ( n + \sqrt{n^2 - 14 n + 113} - 15 )}{8 (n-7)}.
    \]
    For $n \geq 35$, the right-hand side is $\geq \frac{\left(n-1\right)\left(n-11\right)}{4\left(n-7\right)} > \frac{n-7}{4}$, so we obtain $L_k(s_k) > 0$ for $k \leq \frac{1}{4} (n-7)$.

    Next, we consider $\tfrac{1}{5} n \leq k \leq \frac{n - \sqrt{8(n-4)}}{2}$.
    We then have $s_k < \frac{n+2k-8}{4(n-4)} < \hat{s}_2$, from which $L_k(s_k) > 0$ will again follow.
    The second bound is obtained from removing the square root term, and $\frac{n+2k-8}{4(n-4)} \in ( \frac{7}{20}, \frac{1}{2})$ for $k$ in the above range.
    To see that $s_k < \frac{n+2k-8}{4(n-4)}$, we equivalently write $F_{n,k}( \frac{n+2k-8}{4(n-4)}) < 0$.
    For $s = \frac{n+2k-8}{4(n-4)} \in ( \frac{7}{20}, \frac{1}{2})$, the series expansion of ${}_2 F_1 ( \frac{n-1}{2} , - \frac{1}{2} ; \frac{k}{2} ; s)$ is absolutely convergent, and the resulting inequality follows from
    \[
    1 < \sum_{m=0}^{\infty} \frac{(n-1) \cdots (n+2m-1)}{2^{2m+1}k \cdots (k+2m)} \left( \frac{n+2k-8}{4(n-4)}\right)^m 
    \]
    for $\frac{1}{5}n \leq k \leq \frac{n -  \sqrt{8(n-4)}}{2}$.
    A computation analogous to those of Lemmas~\ref{lemma:oscillation-bound} and~\ref{lem:alpha-n,1} shows that the right-hand side is $\geq 1 + C n^{- \frac{1}{2}}$, which occurs when $k = \frac{n - \sqrt{8(n-4)}}{2}$.
    Consequently, $L_k(s_k) > L_k(\hat{s}_2) > 0$.
    This completes the proof, and $4-n \in \cA_{n,k}$ for $1 \leq k \leq \frac{n - \sqrt{8(n-4)}}{2}$.
\end{proof}

\begin{lemma}\label{lemma:4-n.pt.1.5}
    For every $n \geq 7$ and $\frac{1}{3} n \leq k \leq \tfrac{15}{16} n$, we have $4-n \in \cA_{n,k}$.
\end{lemma}
\begin{proof}
    For $7 \leq n \leq 60$, the result $L_{n,k}(s_{n,k}) > 0$ is verified by direct computation; in fact, we find
    \[
    \min_{1 \leq k \leq n-2} L_{n,k} (s_{n,k}) \geq \min_{1 \leq k \leq 5} L_{7,k}(s_{7,k}) > 3 \cdot 10^{-2} > 0,
    \]
    so we may focus on $n \ge 60$ in what follows.
    This allows us to use the bounds from Lemma~\ref{lemma:hypergeometric-zero-estimates} for the root $s_{n,k}$ of the hypergeometric function, which assume the form $s_{n,k} \leq \lambda + \frac{c}{\sqrt{n}}$ for $\lambda = \frac{k}{n}$ and an appropriate constant $c$ depending on $\lambda$.
    By the analysis of Lemma~\ref{lemma:riccati-g-is-positive}, having $L_{n,k}(\lambda + \frac{c}{\sqrt{n}}) > 0$ will imply that $L_{n,k} > 0$ on $[0, \lambda + \frac{c}{\sqrt{n}}]$; in particular, it will give $L_{n,k}(s_{n,k}) > 0$.

    We first set up a general approach.
    As in Lemma~\ref{lemma:hypergeometric-zero-estimates}, we set $s = \lambda + \frac{z}{\sqrt{n}}$ and define $v(z) := \frac{1}{\sqrt{n}} L_k(\lambda + \frac{z}{\sqrt{n}})$, so that $L'(s) = n v'(z)$ and $ns-k = \sqrt{n} z$.
    Moreover,
    \[
    s(1-s) = \lambda(1-\lambda) + (1-\lambda) n^{-1/2} z - n^{-1} z^2 = \lambda(1-\lambda) + O(n^{-1/2})
    \]
    for $n \geq 60$ and uniformly bounded $z$.
    Writing $k=n \lambda$ and $s = \lambda + \frac{z}{\sqrt{n}}$ transforms $P_k(s)$ into
    \begin{align*}
        P_k(s) &= (n \lambda - 1) + (8 - n - 2n \lambda) s + 2(n-4) s^2 \\
        &=\sqrt{n} z (2 \lambda-1) + \bigl( - 1 + 8 \lambda(1-\lambda) \bigr) + O( 1 + R_n + n^{-1/2} R_n^2)
    \end{align*}
    with a uniform remainder error term $R_n(z)$.
    Dividing the transformed equation by $n$ and using the above expansion, we arrive at
    \[
    2 \lambda(1-\lambda) v'(z) + v(z)^2 + z v(z) = R_n(z), \qquad |R_n(z)| \leq C n^{-1/2}( 1 + |v'(z)| + |z| + z^2 )
    \]
    where $C \leq 2$ for $n \geq 60$.
    As $\lambda \in [ \frac{1}{3}, \frac{15}{16}]$ remains on a bounded interval away from the endpoints, we find that $\lambda(1-\lambda) > \frac{1}{50}$ remains uniformly bounded away from $0$.
    Therefore, the equation converges uniformly on compact sets (in the variable $z$) to the limit ODE
    \[
    2 \lambda(1-\lambda) v'_{\lambda}(z) + v_{\lambda}^2(z) + z v_{\lambda}(z) = 0.
    \]
    The initial value $v_{\lambda}(0)$ is obtained as $v_{\lambda}(0) = \lim_{k \to \infty} \frac{L_k(\lambda)}{\sqrt{n}}$.
     Consequently, for $n \geq 60$ and $z \in [ \frac{9}{25}, \frac{15}{25}]$ (meaning that $s-\lambda = O(n^{-1/2})$), the function $L_k( \lambda + \frac{z}{\sqrt{n}})$ is well-approximated by the solution $v_{\lambda}$ of the above equation with matching initial data at the turning point $s = \lambda$, giving
    \begin{equation}\label{eqn:L_k-bound}
        \left| \tfrac{1}{\sqrt{n}} L_k ( \lambda + \tfrac{z}{\sqrt{n}}) - v_{\lambda}(z) \right| \leq 5 n^{-1}, \qquad \text{for } \; n \geq 60, \quad z \in [ \tfrac{9}{25}, \tfrac{15}{25}].
    \end{equation}
    See, for example,~\cite{olver}*{Ch.~11 and Ch.~12}.
    Finally, we introduce the variable $\xi := \frac{z}{\sqrt{2 \lambda(1-\lambda)}}$ and let $u(\xi) := \frac{1}{\sqrt{2\lambda(1-\lambda)}} v_{\lambda}(z)$, so the above equation is transformed into
\[
    u'(\xi) + u(\xi)^2 + \xi u(\xi) = 0.
\]
    The final equation is independent of $\lambda$, and $u(0) = \frac{1}{\sqrt{2\lambda(1-\lambda)}}v_{\lambda}(0)$.
    It is not hard to see that the general solution of this equation given by
    \begin{equation}\label{eqn:u(xi)-soln}
        u \equiv 0 \qquad \text{or} \qquad u(\xi) = \frac{e^{-\xi^2/2}}{B + \int_{-\infty}^{\xi} e^{-r^2/2} \, dr}, \qquad \text{for some } \; B.
    \end{equation}
    Indeed, we may let $u(\xi) = \frac{\tilde{u}'(\xi)}{\tilde{u}(\xi)}$ to transform the Riccati equation into
    \[
    u' + u^2 + \xi u = \tfrac{1}{\tilde{u}} ( \tilde{u}'' + \xi \tilde{u}') = 0,
    \]
    which implies that $\tilde{u}'(\xi) = C e^{-\xi^2/2}$.
    Integrating once more, we obtain $\tilde{u}(\xi) = BC + C \int_{\infty}^{\xi} e^{-r^2/2} \, dr $; the expression~\eqref{eqn:u(xi)-soln} follows.
    The solution of~\eqref{eqn:u(xi)-soln} is uniquely determined by the initial value $u(0)$.
    For fixed $s<\lambda = \frac{k}{n}$, the function $L_k(s)$ exhibits linear behavior as $n \to \infty$: indeed, $L_k(0) = k-1 = n ( \lambda + O(\frac{1}{n}))$, so we may define $y(s) := \frac{1}{n} L_k(s)$ and divide by $n^2$ to obtain
    \[
    y(s)^2 + (s-\lambda)y(s) + O( n^{-1}) = 0 \implies y(s) ( s-\lambda + y(s)) = 0.
    \]
    Since $y(0) = \lambda + O(n^{-1})$, we deduce that $y(s) =\lambda - s+ O(n^{-1})$ and $L_k(s) = n(\lambda-s) + O(n^{-1})$ as $n \to \infty$.
    Consequently, $\frac{1}{\sqrt{n}}L_k( \lambda + \frac{z}{\sqrt{n}}) = - z + O(n^{-3/2})$ for $z<0$; in particular, this implies that $v_{\lambda}(z) \to - z$ as $z\to - \infty$.
    Moreover, $u(\xi) = \frac{1}{\sqrt{2 \lambda(1-\lambda)}} v_{\lambda} (z) \to \frac{-z}{\sqrt{2\lambda(1-\lambda)}} = - \xi$ as $\xi \to -\infty$.
    Since $e^{-\xi^2/2} \to 0$ as $\xi \to - \infty$, the expression~\eqref{eqn:u(xi)-soln} forces $B=0$.
    We therefore obtain $u(0) = \frac{1}{\int_0^{\infty} e^{-r^2/2} \, dr} = \frac{1}{\sqrt{\pi/2}} = \sqrt{\frac{2}{\pi}}$.
    The expression~\eqref{eqn:L_k-bound} now becomes
    \begin{equation}\label{eqn:lk-bound-final}
        \left| \frac{1}{\sqrt{n}} L_k \Bigl( \lambda + \frac{z}{\sqrt{n}} \Bigr) - 2 \lambda(1-\lambda) \frac{\exp\bigl( - \frac{z^2}{4 \lambda(1-\lambda)} \bigr)}{ \int_{-\infty}^z \exp \bigl( - \frac{r^2}{4\lambda(1-\lambda)} \bigr) \,dr } \right| \leq 5 n^{-1} \qquad \text{for } \; n \geq 60
    \end{equation}
    when $z \in [ \frac{9}{25}, \frac{15}{25}]$.
    Consequently, proving that $L_k ( \lambda + \frac{c}{\sqrt{n}}) >0$ for $n \geq 60$ and the specific values of $c$ from Lemma~\ref{lemma:hypergeometric-zero-estimates} will be reduced to having
    \[
    \phi_c( \lambda) := 2 \lambda(1-\lambda) \frac{\exp\bigl( - \frac{c^2}{4 \lambda(1-\lambda)} \bigr)}{ \int_{-\infty}^c \exp \bigl( - \frac{r^2}{4\lambda(1-\lambda)} \bigr) \,dr }, \qquad \phi_c(\lambda) > \frac{9}{100}
    \]
    for $\lambda = \frac{k}{n}$ in the sub-intervals of $[\frac{1}{3}, \frac{15}{16}]$, and the corresponding value of $c \in \{ \frac{3}{5}, \frac{2}{5}, \frac{9}{25} \}$.
    \begin{enumerate}[I.]
        \item For $\lambda = \frac{k}{n} \in [ \frac{1}{3}, \frac{7}{8}]$, we have $c= \frac{3}{5}$ and $\min_{ [ \frac{1}{3}, \frac{7}{8}] } \phi_{\frac{3}{5}} = \phi_{\frac{3}{5}}( \frac{7}{8}) > \frac{91}{1090} > \frac{9}{100}$, as desired.
        \item For $\lambda = \frac{k}{n} \in [ \frac{7}{8}, \frac{9}{10}]$, we have $c = \frac{2}{5}$ and $\min_{[\frac{7}{8}, \frac{9}{10}]} \phi_{\frac{2}{5}} = \phi_{\frac{2}{5}}( \frac{9}{10}) > \frac{1}{10} > \frac{9}{100}$, as desired.
        \item For $\lambda = \frac{k}{n} \in [ \frac{9}{10}, \frac{15}{16}]$, we have $c = \frac{9}{25}$ and $\min_{ [ \frac{9}{10}, \frac{15}{16}] } \phi_{\frac{9}{25}} = \phi_{\frac{9}{25}}( \frac{15}{16}) > \frac{91}{1000} > \frac{9}{100}$, as desired.
    \end{enumerate}
    We conclude that $\phi_c(\lambda) > \frac{9}{100}$ in all of the above cases, whereby~\eqref{eqn:lk-bound-final} implies that $L_k(s_k) > L_k( \lambda + \frac{c_{\lambda}}{\sqrt{n}}) > 0$, as required.

\end{proof}

\subsection{Large \texorpdfstring{$k$}{k}}

It remains to address the case of $ \frac{15}{16} n \leq k \leq n-4$.

\begin{lemma}\label{lemma:4-n.pt.2}
    For every $n \geq 7$ and $\frac{15}{16} n \leq k \leq n-12$, we have $4-n \in \cA_{n,k}$.
\end{lemma}
\begin{proof}
    For $7 \leq n \leq 60$, we verify that $L_k(s_{n,k}) > 0$ by direct computation; in fact, we find that
    \[
    \min_{1 \leq k \leq n-2} L_{n,k} (s_{n,k}) \geq \min_{1 \leq k \leq 5} L_{7,k}(s_{7,k}) > 3 \cdot 10^{-2} > 0,
    \]
    so we may focus on $n \geq 60$ in what follows.
    We denote $d := n-k$, so the result of Lemma~\ref{lemma:oscillation-bound} shows that $s_{n,k} < 1 - \frac{d+1 - \sqrt{2d+1}}{n}$.
    We suppress the dependence of $L_{n,k} , s_{n,k}$ on $n$ in what follows and denote $s_* := 1- \frac{d+1 - \sqrt{2d+1}}{n}$ as in the proof of Lemma~\ref{lemma:oscillation-bound}.
    Applying Lemma~\ref{lemma:riccati-g-is-positive}, we know that $L_k(s_k) > 0$ will follow from $L_k(  s_*) \geq 0$.
    To prove the latter property, we build a two-piece Riccati subsolution for the ODE~\eqref{eqn:L-equation-this-one-promise} satisfied by $L_k$.
    Namely, we define
    \begin{equation}\label{eqn:phi-k-subsolution}
        \phi_k(s) := \begin{cases}
            (k-1) - (n-4) s, & \text{for } \; s \in [ 0, \frac{k}{n} ), \\
          \dfrac{( \sqrt{2d+1}-2) \bigl( Q(s)^{\frac{\sqrt{2d+1}-3}{2}} -1 \bigr)}{ \sqrt{2d+1} - 2 - Q(s)^{\frac{\sqrt{2d+1}-3}{2}} }  , & \text{for } \; s \in [ \frac{k}{n},  s_* ].
        \end{cases}
    \end{equation}
In the above definition, we denoted for brevity
\[
Q(s) := \frac{s_* (1-s)}{(1-s_*)s} = \frac{(1 - \frac{(\sqrt{2d+1}-1)^2}{2n} ) (1-s) }{\frac{(\sqrt{2d+1}-1)^2}{2n} s} = \frac{2n (1 - \frac{(\sqrt{2d+1}-1)^2}{2n}) }{(\sqrt{2d+1}-1)^2} \cdot \frac{ (1-s) }{s}.
\]
In particular, $Q(s)$ is strictly decreasing on $[ \frac{k}{n}, s_*]$, with $Q(s_*) = 1$, and $\phi_k$ is positive on $[ \frac{k}{n}, s_*)$ and strictly decreasing.
To see that $\phi_k(s)$ is a subsolution of the ODE~\eqref{eqn:L-equation-this-one-promise}, consider the operator
\[
\cR[u] := 2s(1-s) u' + u^2 + (ns-k) u + P_k(s), 
\]
associated to equation~\eqref{eqn:L-equation-this-one-promise}, where $P_k(s) := (k-1) + (8 - n- 2k)s + 2(n-4) s^2$, so that $\cR[L_k] = 0$.
The function $\phi_k(s)$ on $[ \frac{k}{n}, s_*]$ is defined as the solution of the ODE
\[
2s(1-s) \phi'_k(s) + \phi_k(s)^2 + ( \sqrt{2d+1}-1) \phi_k(s) + ( \sqrt{2d+1}-2) = 0
\]
with the explicit expression~\eqref{eqn:phi-k-subsolution}.
Since $Q(s_*) = 1$, the above construction additionally ensures that $\phi_k(s_*) = 0$.
The functions $(ns-k)$ and $P_k(s)$ are both positive and strictly increasing on $[\frac{k}{n}, s_*]$; for the latter, simply note that
\[
P'_k(s) = (8-n-2k) + 4(n-4) s = (n-8-2d) + 4(n-4) s > 0
\]
for $d = n-k \leq \frac{n}{16}$ and $s \geq \frac{k}{n} \geq \frac{15}{16}$, once $n \geq 7$.
Therefore, $ns-k \leq ns_* - k = ( \sqrt{2d+1}-1)$ and
\begin{align*}
P_k(s) &\leq P_k(s_*) = ( \sqrt{2d+1} - 2) + \frac{( \sqrt{2d+1}-1)^2 ( 5 - \sqrt{2d+1} )}{n} - \frac{2 ( \sqrt{2d+1}-1)^4}{n^2} \\
&\leq \sqrt{2d+1} - 2,
\end{align*}
because the last term is negative and the middle term is non-positive for $d \geq 12$, so $\sqrt{2d+1} \geq 5$.
We may therefore write
\[
2s(1-s) \phi'_k(s) + \phi_k(s)^2 + (ns_* - k) \phi_k(s) + P_k(s_*) = 0
\]
and the above monotonicity discussion shows that $\cR[\phi_k] < 0$ on $[ \frac{k}{n}, s_*]$, meaning that it is a subsolution there.
On the interval $[0, \frac{k}{n})$, we observe that $\phi'_k(s) = - (n-4)$, so 
\[
2s(1-s) \phi'_k(s) = - 2(n-4) s(1-s), \qquad \phi_k(s) + ns-k = - 1 + 4s,
\]
and so $\phi_k^2 + (ns-k) \phi_k = \phi_k(s) (4s-1)$.
Combining this with $P_k(s) = (k-1)+ (8-n-2k) s + 2(n-4) s^2$, we arrive at
\[
\cR[\phi_k](s) = 2s(k-n+4) \leq - 16 s < 0
\]
by using $k \leq n-12$.
Therefore, $\phi_k$ is a strict subsolution on $[0, \frac{k}{n})$.
Using the construction of $\phi_k$ from~\eqref{eqn:phi-k-subsolution} to evaluate at $s = \frac{k}{n}$, we compute that
\[
Q \Bigl( \frac{k}{n} \Bigr)^{\frac{\sqrt{2d+1}-3}{2}} = \Bigl( 1 + \frac{\sqrt{2d+1}-1}{k} \Bigr)^{\frac{\sqrt{2d+1}-3}{2}} \Bigl( 1 + \frac{2}{\sqrt{2d+1}-1} \Bigr)^{\frac{\sqrt{2d+1}-3}{2}}.
\]
The bound $\log(1+s) \leq s$ implies that $(1+s)^p = \exp( p \log(1+s)) \leq \exp(ps)$.
From $\frac{k}{n} \geq \frac{15}{16}$ we have $d \leq \frac{n}{16}$ and $d \leq \frac{k}{15}$, whereby 
\[
\frac{\sqrt{2d+1}-1}{k} \leq \frac{2}{15 ( \sqrt{2d+1} + 1)} \implies \Bigl( 1 + \frac{\sqrt{2d+1}-1}{k} \Bigr)^{\frac{\sqrt{2d+1}-3}{2}} \leq \exp \left( \frac{\sqrt{2d+1}-3}{15 ( \sqrt{2d+1}+1) } \right).
\]
For $A \geq 4$, the alternating series estimate $\log(1 + s) \leq s - \frac{s^2}{2} + \frac{s^3}{3}$ gives
\[
\frac{A}{2} \log \Bigl( 1 + \frac{2}{A} \Bigr) \leq \frac{A}{2} \Bigl( \frac{2}{A} - \frac{2}{A^2} + \frac{8}{3 A^3} \Bigr) = 1 - \frac{1}{A} + \frac{4}{3A^2}
\]
and exponentiating implies that $( 1 + \frac{2}{A})^{\frac{A-2}{2}} \leq \exp ( 1 - \frac{1}{A} + \frac{4}{3A^2} ) \frac{A}{A+2}$.
Applying this property for $A = \sqrt{2d+1}-1$ and combining the above estimates with $\frac{A-2}{15(A+2)} \leq \frac{1}{15}$, we may further bound
\[
Q \Bigl( \frac{k}{n} \Bigr)^{\frac{\sqrt{2d+1}-3}{2}} \leq \exp \Bigl( \frac{16}{15} - \frac{1}{A} + \frac{4}{3A^2} \Bigr) \frac{A}{A+2}.
\]
Applying the inequality $\log(1-s) \geq - s - s^2$ for $0 \leq s \leq \frac{1}{2}$ to $s = \frac{2}{A(A+1)} \leq \frac{1}{10}$, we also find
\[
\log \Bigl( \frac{(A-1)(A+2)}{A(A+1)} \Bigr) = \log \Bigl( 1 - \frac{2}{A(A+1)} \Bigr) > - \frac{2}{A} - \frac{1}{4A^2}.
\]
meaning that $ \frac{A}{A+2} \leq \frac{A-1}{A+1} \exp( \frac{1}{2A} + \frac{1}{4A^2}) $.
We conclude that
\begin{align*}
    Q \Bigl( \frac{k}{n} \Bigr)^{\frac{\sqrt{2d+1}-3}{2}} &\leq \exp \Bigl( \frac{16}{15} - \frac{1}{A} + \frac{4}{3A^2} \Bigr) \exp \Bigl( \frac{1}{2A} + \frac{1}{4A^2} \Bigr) \frac{A-1}{A+1} \leq \frac{3(A-1)}{A+1}.
\end{align*}
In the last step, we used the fact that $ \frac{16}{15} -\frac{1}{2A} + \frac{19}{12 A^2} \leq \frac{333}{320} < 1.05$ for $A \geq 4$, while $e^{1.05} < \frac{29}{10} < 3$.
Using this inequality in~\eqref{eqn:phi-k-subsolution}, where $A = \sqrt{2d+1} - 1$, we find that
\[
\phi_k \Bigl( \frac{k}{n} \Bigr) \leq \frac{(A-1) ( 3 \frac{A-1}{A+1} - 1) }{ A - 1 - 3 \frac{A-1}{A+1}  } = \frac{2(A-2)}{A-2} = 2.
\]
As a consequence, we may use $\frac{k}{n} \geq \frac{15}{16}$ to obtain
\[
\lim_{ s \nearrow \frac{k}{n} } \phi_k( s) = \tfrac{4k}{n} -1 \geq \tfrac{11}{4} > \phi_k \bigl( \tfrac{k}{n} \bigr).
\]
Therefore, $\phi_k(s)$ has a decreasing jump discontinuity at $\frac{k}{n}$ and is a strict subsolution for $\cR$ on the interval $[0,s_*]$.
Finally, we observe that $\phi_k(0) = L_k(0) = k-1$, while $k \leq n-12$ implies
\[
L'_k(0) = - (n-2) + \tfrac{2(n-4)}{k} > - (n-2) + 2 = - (n-4) = \phi'_k(0),
\]
meaning that $L_k(s)> \phi_k(s)$ for small $s>0$.
Since $\phi_k(s)$ is a strict subsolution on $(0,s_*)$, we conclude that $L_k(s) > \phi_k(s)$; therefore, $L_k(s_*) > \phi_k(s_*) = 0$ completes the proof.
\end{proof}

Finally, we treat the finite set of cases $k \in \{ n-11, n-10, \dots, n-4\}$ for which the bound $s_{n,k} < 1 - \frac{d+1-\sqrt{2d+1}}{n}$ from Lemma~\ref{lemma:oscillation-bound} is insufficient.
The final case is treated by a simple rescaling argument as in Lemmas~\ref{lemma:hypergeometric-zero-estimates} and~\ref{lem:alpha-n,1}.
\begin{lemma}\label{lemma:4-n-final-part}
    For every $n \geq 7$ and $n-11 \leq k \leq n-4$, we have $4-n \in \cA_{n,k}$.
\end{lemma}
\begin{proof}
Since the case $k \leq \frac{15}{16} n$ is treated in Lemma~\ref{lemma:4-n.pt.1.5}, having $k \leq n-4$ requires $n \geq 64$.
For $d = n-k \in \{ 4,\dots, 11\}$ and $n \geq 16 d \geq 64$, the estimates of Lemma~\ref{lemma:oscillation-bound} apply verbatim to show the refinement of the zero-bound by $s_k < s_* :=1 - \frac{2d+1 - 2 \sqrt{2d+1}}{2n}$.
For even $d$, the details are identical to the proof therein, upon directly computing each of the summands for the polynomial $P_m$, where $m = \frac{d}{2}-1 \in \{1,2,3,4\}$, and keeping track of the negative terms in $R_m(s)$, coming from $\Lambda_{\ell}(s)$ for $\ell \geq 1$ as in~\eqref{eqn:lambda-ell(s)}.
For odd $d$, the adaptation is again straightforward by using the expansion~\cite{dtmf}*{15.8.4} for half-integer parameters, where $m = \frac{d}{2}-1 \in \frac{1}{2} \{ 3,5,7,9 \}$.
Applying Lemma~\ref{lemma:riccati-g-is-positive}, we know that $L_k(s_k) > 0$ will follow from $L_k(  s_*) \geq 0$.

As in Lemma~\ref{lemma:4-n.pt.2}, we will prove $L_k(s_*) \geq 0$ by means of a Riccati subsolution $\phi_k$ on $[0,s_*]$.
The operator $\cR$ associated to equation~\eqref{eqn:L-equation-this-one-promise} satisfies
\[
\cR[ (k-1) - (n-4) s] = -2s(n-4-k) \leq 0, \qquad L(0) = (k-1)
\]
as computed above, meaning that $(k-1) - (n-4) s$ is a subsolution for~\eqref{eqn:L-equation-this-one-promise} and $L_k(s) \geq (k-1) - (n-4) s$.
When $k=n-4$, the above relation becomes an equality with $L_{n-4}(s) = (n-5) (1 - \frac{n-4}{n-5} s)$ and $L_{n-4}(s)>0$ for $s<1-\frac{1}{n-4}$.
In particular, $s_*(d=4) = 1 - \frac{3}{2n} \leq 1 - \frac{1}{n-4}$ for $n\geq 12$.
For $k=n-5$, we likewise find $L_{n-5}(s)>0$ for $s < 1 - \frac{2}{n-4}$, and $s_*(d=5) = 1 - \tfrac{11 - 2 \sqrt{11}}{2n} \leq 1- \tfrac{2}{n-4}$ for $n \geq \frac{4(11- 2 \sqrt{11})}{7 - 2 \sqrt{11}} \approx 47.6$, meaning that $L_{n-5}(s_*) > 0$ for $n \geq 64$.

It remains to consider $6 \leq d \leq 11$.
As in~\eqref{eqn:phi-k-subsolution}, we preserve the linear piece $\phi_k(s) = (k-1) - (n-4) s$ on $[ 0, \frac{k}{n}]$ and give a more refined construction on $[ \frac{k}{n}, s_*]$ by solving the majorized Riccati equation obtained by computing the coefficient functions $(ns-k)$ and $P_k(s)$ at the endpoints.
Recall that $P'_k(s)>0$ on $[ \frac{k}{n}, s_*]$, and consider the solution of the backwards ODE on $[ \frac{k}{n}, s_*]$,
\begin{equation}\label{eqn:terminal-riccati-equation}
2 s(1-s) \phi'_k(s) + \phi_k(s)^2 + ( ns_* - k) \phi_k(s) + P_k(s_*) = 0, \qquad \phi_k(s_*) = 0.
\end{equation}
We again denote $Q(s) := \frac{s_*(1-s)}{(1-s_*) s}$ and $A := \sqrt{2d+1}-1$, so $s_* = 1 - \frac{A^2-1}{2n}$ and $Q(s)$ is strictly decreasing on $[\frac{k}{n}, s_*)$, with $Q(s_*) = 1$.
The quadratic equation $r^2 + (A + \frac{1}{2}) r + P_k(s_*) = 0$ has two roots $r_{\pm}$, given by
\begin{equation}\label{eqn:two-roots-quadratic}
r_{\pm} = - \frac{A + \frac{1}{2}}{2} \pm \frac{\sqrt{\Delta}}{2} , \qquad \Delta = \bigl( A + \tfrac{1}{2} \bigr)^2 - 4 P_k(s_*).
\end{equation}
Crucially, we see the quadratic has strictly positive discriminant $\Delta > 1$.
Evaluating as done in Lemma~\ref{lemma:4-n.pt.2}, we find that $P_k(s_*) = A - \tfrac{1}{2} + \tfrac{b_1(d)}{n} + \tfrac{b_2(d)}{n^2}$, where
\begin{align*}
b_1(d) &= - 2 Ad + 11d - 10 A - \tfrac{7}{2}, \qquad b_2(d) = 16Ad - 8d^2 - 8d + 8A-2.
\end{align*}
Using the bound $\sqrt{2d+1} \leq \frac{d}{2}+1$ for $d \geq 6$, hence $A = \sqrt{2d+1}-1 \leq \frac{d}{2}$, we have
\begin{align*}
b_2(d) &\leq 16 \cdot \tfrac{d}{2} \cdot d - 8 d^2 - 8d + 8 \cdot \tfrac{d}{2} - 2 = - 4d-2 < 0, \\
\Delta &= (A + \tfrac{1}{2})^2 - 4 ( A - \tfrac{1}{2}) - 4 \tfrac{b_1(d)}{n} - 4 \tfrac{b_2(d)}{n^2} > (A - \tfrac{3}{2})^2 - 4 \tfrac{b_1(d)}{n} + \tfrac{16d+8}{n^2}.
\end{align*}
Using $A'(d) = \frac{1}{\sqrt{2d+1}}$, we compute
\[
b'_1(d) = - 2 A(d) - 2 d A'(d) + 11 - 10 A'(d) = \tfrac{13 \sqrt{2d+1} - 3(2d+1) - 9}{\sqrt{2d+1}} < 0 
\]
for $d \geq 6$, so $b_1(d)$ is strictly decreasing with 
\[
b_1(6) = - 12 A(6) + 66 - 10 \, A(6) - \tfrac{7}{2} = \tfrac{169}{2} - 22 \sqrt{13} < \tfrac{52}{10}.
\]
Since $A(d) - \frac{3}{2} = \sqrt{2d+1} - \frac{5}{2} \geq \sqrt{13} - \frac{5}{2}$ for $6 \leq d \leq 11$, we conclude that $\Delta(d)$ is strictly decreasing in $d$.
For $n \geq 16 d \geq 96$, we find
\[
\Delta > ( \sqrt{13} - \tfrac{5}{2})^2 - 4 \bigl( \tfrac{169}{2} - 22 \sqrt{13} \bigr) n^{-1} + 104 n^{-2} > \tfrac{101}{100} 
\]
whereby $\Delta>1$.
Therefore, the roots $r_{\pm}$ of the quadratic defined in~\eqref{eqn:two-roots-quadratic} are both real, negative (since $P_k(s_*) > 0$ makes $\sqrt{\Delta} < A + \frac{1}{2}$) and $r_- < r_+ < 0$.
Finally, we produce the barrier function
\begin{equation}\label{eqn:phi-k-subsolution-special-case}
        \phi_k(s) := \begin{cases}
            (k-1) - (n-4) s, & \text{for } \; s \in [ 0, \frac{k}{n} ), \\
          \dfrac{r_+ r_- ( Q(s)^{\frac{\sqrt{\Delta}}{2}} - 1) }{r_+ - r_- Q(s)^{\frac{\sqrt{\Delta}}{2}}}  , & \text{for } \; s \in [ \frac{k}{n},  s_* ].
        \end{cases}
    \end{equation}
    By definition, the function $\phi_k(s)$ is the solution of~\eqref{eqn:terminal-riccati-equation} on $[ \frac{k}{n}, s_*]$, and generalizes the construction ~\eqref{eqn:terminal-riccati-equation} to our setting.
    Recall that $Q(s)$ is strictly decreasing on $[ \frac{k}{n}, s_*]$, with $Q(s_*) = 1$, while $r_- < r_+ < 0$; therefore, $r_+ - r_- Q(s)^{\frac{\sqrt{\Delta}}2{}} \geq r_+ - r_- > 0$ implies that $\phi_k(s)$ is a well-defined function with $\phi_k(s) \geq 0$ and $\phi'_k<0$.
    The increasing property of the coefficients, combined with~\eqref{eqn:terminal-riccati-equation}, allows us to conclude that $\cR[\phi_k] < 0$ as in Lemma~\ref{lemma:4-n.pt.2}.
    Since $\cR[\phi_k] < 0$ on $[ 0 , \frac{k}{n}]$ was proved above, we deduce that $\phi_k$ is a piecewise subsolution.
    Moreover,
    \[
    \phi_k(s) + r_+ = \frac{r_+ r_- \bigl( Q(s)^{\frac{\sqrt{\Delta}}{2}} - 1 \bigr) + r_+ \bigl( r_+- r_- Q(s)^{\frac{\sqrt{\Delta}}{2}} \bigr) }{ r_+ - r_- Q(s)^{\frac{\sqrt{\Delta}}{2}} } = \frac{r_+ (r_+ - r_- )}{r_+ - r_- Q(s)^{\frac{\sqrt{\Delta}}{2}}} < 0
    \]
    due to $r_+ < 0$ and $r_+ - r_- Q(s)^{\frac{\sqrt{\Delta}}{2}} \geq r_+ - r_- > 0$ as above.
    Consequently,
    \[
    \phi_k(s) < - r_+ = \frac{(A+ \frac{1}{2}) - \sqrt{\Delta}}{2} \leq \frac{\sqrt{2d+1} - \frac{1}{2}}{2}
    \]
    by using $A = \sqrt{2d+1}-1$ and $\frac{\sqrt{2d+1} - \frac{1}{2}}{2} < \frac{11}{5}$ for $d \leq 11$.
    We therefore obtain
    \[
    \lim_{s \nearrow \frac{k}{n}} \phi_k(s) = \tfrac{4k}{n} - 1 = 3 - \tfrac{4d}{n} \geq \tfrac{11}{5} > \phi_k ( \tfrac{k}{n})
    \]
    by rearranging the inequality into $n \geq 5d$, valid for $n \geq 64$ and $d \leq 11$.
    We conclude that $\phi_k(s)$ has a decreasing jump discontinuity at $\frac{k}{n}$ and is a strict subsolution for $\cR$ on the interval $[0,s_*]$.
Finally, we observe that $\phi_k(0) = L_k(0) = k-1$, while $6 \leq d \leq 11$ implies
\[
L'_k(0) = - (n-2) + \tfrac{2(n-4)}{k} > - (n-2) + 2 = - (n-4) = \phi'_k(0),
\]
meaning that $L_k(s)> \phi_k(s)$ for small $s>0$.
Since $\phi_k(s)$ is a strict subsolution on $(0,s_*)$, we conclude that $L_k(s) > \phi_k(s)$; therefore, $L_k(s_*) > \phi_k(s_*) = 0$ completes the proof.
\end{proof}

\section{Proof of the Main Results}

We now collect the results to complete the proofs of our main Theorems.

\begin{corollary}\label{prop:4-n}
    For every $n \geq 7$ and $1 \leq k \leq n-2$, we have $4-n \in \cA_{n,k}$.
\end{corollary}
\begin{proof}
This result follows from combining Lemmas~\ref{lem:alpha-n,1}, ~\ref{lemma:4-n.pt.1}, ~\ref{lemma:4-n.pt.1.5},~\ref{lemma:4-n.pt.2} and~\ref{lemma:4-n-final-part}, since they prove that $4 -n \in \cA_{n,k}$ for $n \geq 7$ and for
    \[
    k \in \{ 1,2, n-3, n-2\}, \qquad 3 \leq k \leq \tfrac{n - \sqrt{8(n-4)}}{2}, \qquad \tfrac{1}{3} n \leq k \leq \tfrac{15}{16} n, \qquad \tfrac{15}{16} n \leq k \leq n-4.
    \]
    For $n \geq 70$, we have $\frac{n}{3} < \frac{n - \sqrt{8(n-4)}}{2}$, so this covers all $1 \leq k \leq n-2$.
    For $n < 70$ and intermediate $k$, the property $L_k(s_k) > 0$ is verified to hold by direct computation, whereby $4-n \in \cA_{n,k}$ for all $1 \leq k \leq n-2$, as desired.
\end{proof}

\begin{proof}[Proof of Proposition~\ref{prop:strict-subsolution}]
    The function $V_{n,k} = \rho^{4-n} g_{n,k,4-n}(t)$ is harmonic, by construction, and Corollary~\ref{prop:4-n} shows that the stability condition~\eqref{eqn:strict-stability-inequality} is satisfied strictly for every $n \geq 7$ and $1 \leq k \leq n-2$, with $\alpha = 4-n$.
    The computations~\eqref{eqn:Robin-problem}, \eqref{eqn:mean-curvature-computation}, and~\eqref{eqn:sign-boundary-term} therefore imply that $V_{n,k}$ satisfies the strict inequality $\partial_{\nu} V_{n,k} + H V_{n,k} < 0$, so it is a subsolution as claimed.
\end{proof}

\begin{proof}[Proof of Theorem~\ref{thm:stability-of-one-phase-cones}]
    For $3 \leq n \leq 6$, the ten cones $U_{n,k}$ are all unstable by the computations of Caffarelli-Jerison-Kenig and Hong~\cites{cjk-2, hong-singular}.
    For every $n \geq 7$ and $1 \leq k \leq n-2$, Proposition~\ref{prop:strict-subsolution} shows that the inequality
    \[
    \frac{g'_{n,k,\alpha}(t_{n,k})}{g_{n,k,\alpha}(t_{n,k})} > \frac{(n-2) t_{n,k} - (k-1) t_{n,k}^{-1}}{1 - t_{n,k}^2}
    \]
    holds strictly, with $\alpha = 4-n$.
    By Lemma~\ref{lemma:interval-of-alpha}, we conclude that this condition is also satisfied with strict inequality for $\alpha = \frac{2-n}{2}$; this implies the strict stability of $U_{n,k}$.
\end{proof}

\begin{proof}[Proof of Proposition~\ref{prop:first-eigenvalue-bound}]
    The discussion of Section~\ref{subsection:stability-of-solutions} implies that the harmonic function $\varphi = \rho^{\alpha} g(\omega)$ is a subsolution of the linearized problem~\eqref{eqn:linearized-problem} if and only if $\alpha \in ( \gamma_-, \gamma_+) $ is inside the indicial interval as in~\eqref{eqn:degree-of-homogeneity}.
    By Proposition~\ref{prop:4-n}, $\alpha = 4-n$ satisfies this property, implying the inequality $4-n > \gamma_- = - \frac{n-2}{2} - \sqrt{( \frac{n-2}{2})^2 + \lambda_1}$.
    For $n \geq 7$, this may be rearranged into
    \begin{align*}
        2(n-4) < n-2 + \sqrt{ (n-2)^2 + 4 \lambda_1 } & \iff (n-6)^2 < (n-2)^2 + 4 \lambda_1 \\
        & \iff 4(8-2n) < 4 \lambda_1, 
    \end{align*}
    which implies~\eqref{eqn:lambda-1-n-k-bound}.
    For the second inequality, note that $\gamma_+ + \gamma_- = 2-n$, hence $2-n < \gamma_- < 4-n$ proves that $(4-n,-2) \subset (\gamma_- , \gamma_+) \subset (2-n,0)$.
    Finally, the discussion of~\ref{subsection:stability-of-solutions} shows that solutions of~\eqref{eqn:linearized-problem} decompose into homogeneous Jacobi fields with degrees determined by the Robin spectrum on the link, and no Jacobi field can have growth rate in an interval free of indicial roots, cf.~\cite{one-phase-simon-solomon}*{\S~2}.
    Therefore, there do not exist Jacobi fields with decay $\sim \rho^{\sigma}$ for $\sigma \in (4-n,-2)$, completing the proof.
\end{proof}

Finally we present some sample computations of the quantity $\lambda_{n,k} := |\lambda_1(n,k)|$ on the cones $U_{n,k}$, providing evidence towards Conjecture~\ref{conj:eigenvalue}.
The framework of Proposition~\ref{prop:spectral-analysis} allows us to compute $\lambda_{n,k}$ by solving the spectral problem~\eqref{eqn:eigenvalue-condition} for $p=q=0$, i.e.,
\[
(1-t^2) \Phi'' + \Bigl( \frac{k-1}{t} - (n-1) t \Bigr) \Phi' = \lambda_{n,k} \Phi, \qquad \frac{\Phi'(t_{n,k})}{\Phi(t_{n,k})} = \frac{(n-2) t_{n,k} - (k-1) t_{n,k}^{-1}}{1 - t_{n,k}^2} \, .
\]
Computing $t_{n,k}$ and then solving this equation, we compute various values of $\lambda_1(n,k)$ and the corresponding homogeneity $\gamma_+(n,k)$ in Table~\ref{tab:eigenvalues}.

\begin{table}
    \centering
   
    \begin{tabular}{|c|c||c|c|c|c|c|c||c|c|c|}
    \cline{1-5}\cline{7-11}
    $n$ & $k$ & $t_{n,k}$ & $-\lambda_1$ & $-\gamma_+$ && $n$ & $k$ & $t_{n,k}$ & $-\lambda_1$ & $-\gamma_+$ \\ \cline{1-5}\cline{7-11} \cline{1-5}\cline{7-11}
    7 & 1 & 0.52 & 5.698 & 1.757  && 7 & 2 & 0.69 & 5.639 & 1.718 \\ \cline{1-5}\cline{7-11}
    7 & 3 & 0.81 & 5.607 & 1.698  && 7 & 4 & 0.89 & 5.581 & 1.682 \\ \cline{1-5}\cline{7-11}
    7 & 5 & 0.96 & 5.551 & 1.664  & \multicolumn{5}{c}{} \\
    
    \cline{1-5}
    \multicolumn{11}{c}{} \\[1pt]
    \cline{1-5}\cline{7-11}
    8 & 1 & 0.48 & 6.699 & 1.483  && 8 & 2 & 0.65 & 6.642 & 1.464 \\ \cline{1-5}\cline{7-11}
    8 & 3 & 0.78 & 6.613 & 1.455  && 8 & 4 & 0.87 & 6.591 & 1.448 \\ \cline{1-5}\cline{7-11}
    8 & 5 & 0.91 & 6.571 & 1.441  && 8 & 6 & 0.97 & 6.544 & 1.433 \\
    
    \cline{1-5}\cline{7-11}
    \multicolumn{11}{c}{} \\[1pt]
    \cline{1-5}\cline{7-11}
    
    9 & 1 & 0.45 & 7.701 & 1.367  && 9 & 2 & 0.61 & 7.645 & 1.354 \\ \cline{1-5}\cline{7-11}
    9 & 3 & 0.75 & 7.618 & 1.348  && 9 & 4 & 0.84 & 7.599 & 1.343 \\ \cline{1-5}\cline{7-11}
    9 & 5 & 0.89 & 7.582 & 1.339  && 9 & 6 & 0.94 & 7.564 & 1.335 \\ \cline{1-5}\cline{7-11}
    9 & 7 & 0.98 & 7.540 & 1.330  & \multicolumn{5}{c}{} \\
    
    \cline{1-5}
    \multicolumn{11}{c}{} \\[1pt]
    \cline{1-5}\cline{7-11}
    
    10 & 1 & 0.43 & 8.702 & 1.298  && 10 & 2 & 0.57 & 8.647 & 1.288 \\ \cline{1-5}\cline{7-11}
    10 & 3 & 0.68 & 8.621 & 1.283  && 10 & 4 & 0.76 & 8.604 & 1.280 \\ \cline{1-5}\cline{7-11}
    10 & 5 & 0.83 & 8.589 & 1.278  && 10 & 6 & 0.89 & 8.575 & 1.275 \\ \cline{1-5}\cline{7-11}
    10 & 7 & 0.94 & 8.559 & 1.272  && 10 & 8 & 0.98 & 8.536 & 1.228 \\
    
    \cline{1-5}\cline{7-11}
    \multicolumn{11}{c}{} \\[1pt]
    \cline{1-5}\cline{7-11}
    
    11 & 1 & 0.41 & 9.702 & 1.252  && 11 & 2 & 0.55 & 9.649 & 1.244 \\ \cline{1-5}\cline{7-11}
    11 & 3 & 0.65 & 9.624 & 1.240  && 11 & 4 & 0.72 & 9.607 & 1.238 \\ \cline{1-5}\cline{7-11}
    11 & 5 & 0.79 & 9.594 & 1.236  && 11 & 6 & 0.85 & 9.582 & 1.234 \\ \cline{1-5}\cline{7-11}
    11 & 7 & 0.90 & 9.570 & 1.232  && 11 & 8 & 0.94 & 9.555 & 1.230 \\ \cline{1-5}\cline{7-11}
    11 & 9 & 0.98 & 9.533 & 1.226  & \multicolumn{5}{c}{} \\
    \cline{1-5}
    \multicolumn{11}{c}{} \\[1pt]
    \cline{1-5}\cline{7-11}
    
    12 & 1 & 0.38 & 10.703 & 1.219  && 12 & 2 & 0.53 & 10.650 & 1.212 \\ \cline{1-5}\cline{7-11}
    12 & 3 & 0.67 & 10.626 & 1.209  && 12 & 4 & 0.69 & 10.610 & 1.207 \\ \cline{1-5}\cline{7-11}
    12 & 5 & 0.76 & 10.598 & 1.205  && 12 & 6 & 0.82 & 10.587 & 1.204 \\ \cline{1-5}\cline{7-11}
    12 & 7 & 0.87 & 10.577 & 1.202  && 12 & 8 & 0.91 & 10.566 & 1.201 \\ \cline{1-5}\cline{7-11}
    12 & 9 & 0.95 & 10.552 & 1.199  && 12 & 10 & 0.98 & 10.531 & 1.196 \\
    \cline{1-5}\cline{7-11}
    \end{tabular}

    \caption{Numerically computed approximate values for the zero $t_{n,k}$ and the quantities $-\lambda_1(n,k)$ and $- \gamma_+(n,k)$ for the cones $U_{n,k}$, giving evidence for Conjecture~\ref{conj:eigenvalue}.}
    \label{tab:eigenvalues}
\end{table}

Further analysis of hypergeometric functions could be applied to prove that the quantity $\lambda_{n,k}$ is decreasing in $k$, and minimized by $\lambda_{n,n-2}$, providing evidence towards Conjecture~\ref{conj:eigenvalue}.

\nocite{*}
\bibliography{ref}

\end{document}